\documentclass[11pt]{amsart}

\usepackage[a4paper,margin=3cm]{geometry}

\usepackage[T1]{fontenc}
\usepackage[utf8]{inputenc}
\usepackage{lmodern}

\usepackage{amsmath,amssymb,amsfonts,amsthm}
\usepackage{mathtools}
\usepackage{mathrsfs}
\usepackage{bm}

\usepackage{tikz-cd}

\usepackage{enumitem}
\usepackage{cite}

\usepackage[colorlinks=true,
            linkcolor=blue,
            citecolor=blue,
            urlcolor=blue]{hyperref}

\usepackage{doi}

\theoremstyle{plain}
\newtheorem{theorem}{Theorem}[section]
\newtheorem{proposition}[theorem]{Proposition}
\newtheorem{lemma}[theorem]{Lemma}
\newtheorem{corollary}[theorem]{Corollary}

\theoremstyle{definition}
\newtheorem{definition}[theorem]{Definition}
\newtheorem{assumption}[theorem]{Assumption}

\theoremstyle{remark}
\newtheorem{remark}[theorem]{Remark}

\newcommand{\Hilb}{\mathcal H}

\newcommand{\Tr}{\operatorname{Tr}}
\newcommand{\FP}{\operatorname{FP}}
\newcommand{\Id}{\mathbf 1}
\newcommand{\ii}{\mathrm i}
\newcommand{\ee}{\mathrm e}
\newcommand{\dd}{\,\mathrm d}

\title[Spectral cut-off oscillatory integrals]
{Spectral Cut-off Oscillatory Integrals for Non-Autonomous Hamiltonian Evolution Equations}

\author{Jean-Pierre Magnot}
\address{$^*$  Univ. Angers, CNRS, LAREMA, SFR MATHSTIC, F-49000 Angers, France ;\\ 
Lepage Research Institute, 17 novembra 1, 081 16 Presov, Slovakia;
	\\ and \\  Lyc\'ee Jeanne d'Arc, \\ Avenue de Grande Bretagne, \\ 63000 Clermont-Ferrand, France}
    \email{jean-pierr.magnot@ac-clermont.fr; magnot@math.cnrs.fr}
\date{}



\begin{document}

\begin{abstract}
We develop a spectral cut-off and time-slicing construction for non-autonomous Hamiltonian evolution equations. Let \(H_0\) be a positive self-adjoint reference operator with compact resolvent on a Hilbert space \(\Hilb\), and let
$$
P_N=\mathbf{1}{[0,N]}(H_0).
$$
For a time-dependent family of generally unbounded symmetric Hamiltonians (H(t)), we consider the finite-dimensional cut-off Hamiltonians
$$
H_N(t)=P_NH(t)P_N.
$$
Their time-sliced propagators admit finite-dimensional state-sum representations and, when suitable configuration-space or phase-space kernels are available, oscillatory integral realizations. We establish commutator conditions implying uniform stability of the cut-off dynamics and construct the full unitary propagator as the strong limit of the finite-dimensional propagators. Additional \(H_0\)-regularity yields the quantitative estimate
$$
\sup{s,t\in I}
|U_N(t,s)P_Nu-U(t,s)u|
\leq
C(1+N)^{-\sigma}|u|{\mu+\sigma}.
$$
For Hamiltonians that are H"older continuous of exponent \(\alpha\) in time, we also prove the joint spectral and time-slicing bound
$$
|U{N,M}(t,s)P_Nu-U(t,s)u|
\leq
C\left((1+N)^{-\sigma}+(1+N)^\mu M^{-\alpha}\right)
|u|_{\mu+\sigma}.
$$
The assumptions are verified for time-dependent Schr"odinger operators and symmetric first-order pseudodifferential Hamiltonians. In the periodic case, the construction is compatible with finite-order Floquet--Magnus coefficients for unbounded operators.
\end{abstract}
\maketitle

\section{Introduction}
\label{sec:introduction}

Real-time Hamiltonian evolution is one of the classical meeting points
between functional analysis, partial differential equations and quantum
mechanics.  In its simplest autonomous form, the Schr\"odinger equation
is generated by a self-adjoint operator through Stone's theorem.  In
non-autonomous situations, however, the equation
\[
        \ii\partial_t u(t)=H(t)u(t)
\]
raises a more delicate problem: the family \(H(t)\) is typically
unbounded, its domain may depend on \(t\), and the formal expression
\[
        \mathcal T\exp\left(-\ii\int_s^t H(\tau)\dd\tau\right)
\]
does not by itself define a propagator.  The foundational analytic
background for this question goes back to the theory of evolution
equations in Banach and Hilbert spaces, in particular to the work of
Kato on linear evolution equations and hyperbolic-type equations
\cite{Kato1953,Kato1956,Kato1970,Kato1973,KatoBook}, to Kisy\'nski's
work on abstract Cauchy problems \cite{Kisynski1964}, and to the
semigroup and evolution-equation frameworks developed in
\cite{Yosida1980,Tanabe1979,Pazy1983}.  We shall also use the standard
spectral and self-adjointness background of \cite{ReedSimon1,
ReedSimon2,ReedSimon4}, together with the quadratic form viewpoint of
Simon \cite{Simon1971,Simon1971Annals}.

A second source of motivation is the real-time Feynman path integral.
At the heuristic level, the solution of a Hamiltonian evolution equation
is expected to be represented by an oscillatory integral over paths,
with phase given by the classical action.  This point of view has a long
history, beginning with rigorous formulations of Feynman integrals and
their relation to Schr\"odinger equations
\cite{Nelson1964,AlbeverioHoeghKrohn1976}.  Later developments include
the monograph treatment \cite{AlbeverioHoeghKrohnMazzucchi2008,
Mazzucchi2009}, trace formulae for Schr\"odinger operators
\cite{AlbeverioBlanchardHoeghKrohn1982,
AlbeverioBoutetBrzezniak1996}, finite-dimensional approximation
schemes for oscillatory integrals in infinite dimensions
\cite{AlbeverioBrzezniak1993,AlbeverioBrzezniak1995}, and rigorous
time-slicing approaches
\cite{Fujiwara1979,Fujiwara1980,Fujiwara2017}.  Polynomially growing
and time-dependent potentials have been studied by Feynman integral
methods in \cite{AlbeverioMazzucchi2005,
AlbeverioMazzucchi2005CRAS,AlbeverioMazzucchi2006}, and magnetic
Hamiltonians in \cite{AlbeverioCangiottiMazzucchi2019}.  These works
show that oscillatory representations may be made rigorous in many
important situations, but also that infinite-dimensional oscillatory
measures should not be treated as ordinary measures.

The purpose of the present paper is to develop a functional analytic
construction of real-time oscillatory amplitudes by means of spectral
cut-offs.  Rather than postulating an infinite-dimensional oscillatory
integral, we begin with a positive self-adjoint reference operator
\(H_0\) on a Hilbert space \(\Hilb\), assume for simplicity that
\(H_0\) has compact resolvent, and introduce the spectral projections
\[
        P_N=\mathbf 1_{[0,N]}(H_0).
\]
For a time-dependent Hamiltonian \(H(t)\), we define the finite
dimensional cut-off Hamiltonians
\[
        H_N(t)=P_NH(t)P_N
\]
acting on \(P_N\Hilb\).  For each \(N\), the associated propagator
\(U_N(t,s)\) is generated by a finite-dimensional Schr\"odinger
equation.  Time slicing gives a finite-dimensional state-sum
representation of this propagator.  In realizations for which the
one-step kernels admit suitable phases and amplitudes, the same
construction yields genuine finite-dimensional oscillatory integrals.
The solution of the original equation is then obtained, under suitable
assumptions, as the strong limit
\[
        U(t,s)u
        =
        \lim_{N\to\infty}\lim_{M\to\infty}
        U_{N,M}(t,s)P_Nu,
\]
where \(U_{N,M}(t,s)\) denotes the time-sliced cut-off propagator.  When
an oscillatory realization exists, one may equivalently write
\(U_{N,M}(t,s)=I_{N,M}(t,s)\).

The key point is that the convergence problem is transferred from
measure theory on path space to strong convergence of propagators.  The
proof uses Duhamel's formula and estimates relative to the scale
defined by \(H_0\).  This connects naturally with the classical
Trotter--Kato philosophy and product formulae
\cite{Trotter1959,Chernoff1968,Kurtz1969,IchinoseTamura1997}.  It also
fits the modern theory of non-autonomous Schr\"odinger equations with
time-dependent Hamiltonians, in particular in the case of constant form
domains and stability estimates for time-dependent quantum systems
\cite{Balmaseda2022,Balmaseda2024}.

The spectral cut-off construction is not tied to a particular
coordinate representation.  It applies to several standard classes of
Hamiltonians once the appropriate \(H_0\)-Sobolev scale is chosen.  On
compact manifolds, differential and pseudodifferential Hamiltonians are
controlled by elliptic reference operators.  We shall use the standard
pseudodifferential and microlocal tools of
\cite{HormanderI,HormanderIII,Shubin1987,Taylor1981,Taylor1991}, the
complex powers of elliptic operators introduced by Seeley
\cite{Seeley1967}, and the semiclassical background of
\cite{Robert1987,DimassiSjostrand1999}.  The spectral geometry of
elliptic operators and wave traces, as developed for example in
\cite{DuistermaatGuillemin1975,GuilleminSternberg1977,
Melrose1982,Melrose1995}, provides an important context for later trace
interpretations.  Dirac-type and matrix-valued Hamiltonians are also
included, with the geometric background of
\cite{LawsonMichelsohn1989,BaerBallmann2016}; quadratic Hamiltonians and
harmonic oscillator cut-offs are naturally related to phase-space
analysis and the metaplectic formalism \cite{Folland1989}.

One of the motivations for using spectral cut-offs is that they provide
a common framework for finite-dimensional state-sums, oscillatory
realizations, Galerkin approximations and renormalized trace procedures.
Although the main part of the paper concerns strong convergence of propagators,
we also discuss, in an appendix, how one may use the same cut-offs to
define finite or renormalized real-time amplitudes.  This connects with
the theory of residues and regularized traces for pseudodifferential
operators, including the noncommutative residue
\cite{Wodzicki1984}, the determinant and canonical trace constructions
of Kontsevich and Vishik \cite{KontsevichVishik1994}, log-polyhomogeneous
symbols \cite{Lesch1999}, and regularized integrals and traces
\cite{PaychaScott2007,Paycha2012}.  Conceptually, this is also close to
the use of normalized means for infinite-volume objects and heuristic
infinite-dimensional Lebesgue measures developed in
\cite{Magnot2017MeanValue}.  In the present paper, however, these
renormalized traces are not used to prove the main convergence theorem;
they are presented as a natural continuation of the cut-off
construction.

A further motivation comes from periodically driven quantum systems.
If \(H(t)\) is periodic in time, the high-frequency regime is often
described by effective Hamiltonians.  In finite dimension, this leads to
the Floquet--Magnus expansion, whose origin goes back to Magnus
\cite{Magnus1954} and whose modern formulations and applications are
surveyed in \cite{BlanesCasasOteoRos2009}.  In physics, effective
Hamiltonians are central in Floquet engineering and periodically driven
systems \cite{GoldmanDalibard2014,BukovDAlessioPolkovnikov2015}.
Recent estimates for the Floquet--Magnus expansion and applications to
the quantum Rabi model are given in
\cite{DeyLonigroYuasaBurgarth2025}.  The unbounded case is much more
delicate, since commutators of unbounded operators may fail to be
densely defined or self-adjoint.  The recent work
\cite{BurgarthHillierLonigroRichter2026} develops a rigorous
Floquet--Magnus theory for a broad class of unbounded Hamiltonians.  Our
spectral cut-off approach is complementary: at each finite cut-off one
has a classical finite-dimensional Floquet--Magnus expansion, and one
can then study the convergence of the finite-order effective
Hamiltonians as the cut-off is removed.

We now describe the main contributions of the paper.

First, we construct finite-dimensional state-sum amplitudes by combining
spectral cut-offs with time slicing.  The
spectral parameter \(N\) truncates the energy scale associated with
\(H_0\), while the time-slicing parameter \(M\) gives a product formula
for the cut-off propagator.  These state-sums are defined in every
finite-dimensional cut-off space.  When a configuration-space,
phase-space, or coherent-state realization is available, they give
finite-dimensional oscillatory integrals with specified one-step kernels.

Second, under \(H_0\)-relative assumptions on \(H(t)\), we prove strong
convergence of the cut-off propagators:
\[
        U_N(t,s)P_Nu\longrightarrow U(t,s)u
\]
uniformly for \((t,s)\) in compact time intervals.  The proof is based
on the comparison between the projected exact dynamics
\(P_NU(t,s)u\) and the finite-dimensional cut-off dynamics
\(U_N(t,s)P_Nu\).  The only error term is
\[
        P_NH(t)(I-P_N)U(t,s)u,
\]
which tends to zero because the high-energy tail of the exact solution
vanishes in the required \(H_0\)-Sobolev norm.  Additional regularity
gives the quantitative estimate
\[
        \sup_{s,t\in I}
        \|U_N(t,s)P_Nu-U(t,s)u\|
        \leq
        C(1+N)^{-\sigma}\|u\|_{\mu+\sigma}.
\]
We also give a commutator criterion implying uniform stability of the
cut-off propagators.  Under this criterion, the full propagator need not
be assumed in advance: it is constructed as the uniform strong limit of
the finite-dimensional propagators.  If \(H(t)\) is H\"older continuous
of exponent \(\alpha\) in time, the spectral and time-slicing cut-offs
may be removed jointly, with
\[
\begin{aligned}
        &\|U_{N,M}(t,s)P_Nu-U(t,s)u\| \\
        &\qquad\leq
        C\left((1+N)^{-\sigma}+(1+N)^\mu M^{-\alpha}\right)
        \|u\|_{\mu+\sigma}.
\end{aligned}
\]
Thus any choice satisfying
\((1+N)^\mu M(N)^{-\alpha}\to0\) gives a joint approximation.

Third, in the periodic case, we study the interaction between spectral
cut-offs and finite-dimensional Floquet--Magnus theory.  For each
cut-off \(N\), the coefficients
\[
        H_{\mathrm{FM},N}^{[\ell]}
\]
are finite-dimensional matrices.  We prove that, on smooth vectors,
these coefficients converge to the formal unbounded
Floquet--Magnus coefficients:
\[
        H_{\mathrm{FM},N}^{[\ell]}P_Nu
        \longrightarrow
        H_{\mathrm{FM}}^{[\ell]}u.
\]
Thus the cut-off construction is compatible with finite-order effective
Hamiltonians.

Fourth, we give a family of examples showing that the abstract
hypotheses are natural.  For Schr\"odinger operators with time-dependent
potentials and symmetric first-order pseudodifferential Hamiltonians,
we verify the commutator estimates and obtain construction, Sobolev
stability, and quantitative convergence of the spectral approximations.
The framework also covers higher-order pseudodifferential Hamiltonians
compatible with the reference scale, Dirac-type systems, quadratic
Hamiltonians controlled by the harmonic oscillator, log-polyhomogeneous
pseudodifferential operators, smoothing operators, and abstract
Kato--Rellich perturbations.

The paper is organized as follows.  Section~\ref{sec:cutoff-oscillatory-integrals}
introduces spectral cut-off oscillatory integrals and the associated
time-sliced finite-dimensional propagators.  Section~\ref{sec:convergence-cutoff-dynamics}
proves qualitative and quantitative convergence, establishes uniform
stability from commutator estimates, constructs the full propagator, and
treats the joint removal of the spectral and temporal cut-offs.
Section~\ref{sec:periodic-effective-hamiltonians} treats periodic
Hamiltonians and finite-order effective Hamiltonians.  Section~\ref{sec:examples}
presents the main examples.  Appendix~\ref{app:effective-hamiltonians}
recalls the role of effective Hamiltonians in periodically driven
quantum systems, while Appendix~\ref{app:cutoff-traces} explains how
spectral cut-offs may be used to define finite-part traces and
renormalized real-time amplitudes. Appendix~\ref{app:effective-hamiltonians}
recalls the role of effective Hamiltonians in periodically driven
quantum systems, while Appendix~\ref{app:cutoff-traces} explains how
spectral cut-offs may be used to define finite-part traces and
renormalized real-time amplitudes.

\section{Spectral cut-off oscillatory integrals}
\label{sec:cutoff-oscillatory-integrals}

The aim of this section is to introduce the finite-dimensional
time-sliced objects from which the solutions will be constructed.  The
basic idea is simple: instead of postulating an infinite-dimensional
Feynman integral, we first project the Hamiltonian evolution onto
finite-dimensional spectral subspaces of a reference operator.  On
each such subspace, the propagator admits an exact finite-dimensional
state-sum representation.  When an additional configuration-space,
phase-space, or coherent-state realization is available, the same
time-sliced operators may also admit oscillatory integral
representations.  The full infinite-dimensional solution will then be
recovered by removing the spectral and temporal cut-offs.

\subsection{Reference operator and spectral cut-offs}

Let \(\Hilb\) be a complex separable Hilbert space.  We fix a positive
self-adjoint operator
\[
        H_0:D(H_0)\subset \Hilb\longrightarrow \Hilb.
\]
Throughout this section we assume that \(H_0\) has compact resolvent.
Thus its spectrum consists of a non-decreasing sequence of eigenvalues,
repeated according to multiplicity,
\[
        0\leq \lambda_1\leq \lambda_2\leq \cdots,
        \qquad
        \lambda_j\to+\infty,
\]
and \(\Hilb\) admits an orthonormal basis of eigenvectors of \(H_0\).

For \(N>0\), we denote by
\[
        P_N=\mathbf 1_{[0,N]}(H_0)
\]
the spectral projection of \(H_0\) associated with the interval
\([0,N]\).  We set
\[
        \Hilb_N=P_N\Hilb.
\]
Since \(H_0\) has compact resolvent, \(\Hilb_N\) is finite-dimensional.
The projections \(P_N\) satisfy
\[
        P_Nu\longrightarrow u
\]
strongly in \(\Hilb\) as \(N\to+\infty\).

The reference operator \(H_0\) also defines a scale of Hilbert spaces.
For \(s\geq0\), we set
\[
        \Hilb^s=D\big((1+H_0)^s\big),
        \qquad
        \|u\|_s=\|(1+H_0)^s u\|.
\]
We also write
\[
        \Hilb^\infty
        =
        C^\infty(H_0)
        =
        \bigcap_{s\geq0}\Hilb^s.
\]
The space \(\Hilb^\infty\) will play the role of a common smooth core.

\begin{remark}
The compact-resolvent assumption is natural for the construction of
finite-dimensional spectral cut-offs.  It can be relaxed if one replaces
the projections \(P_N\) by suitable finite-rank approximations or by
localized spectral cut-offs.  We keep the compact-resolvent setting in
order to isolate the main mechanism.
\end{remark}

\subsection{Cut-off Hamiltonians}

Let \(I\subset\mathbb R\) be a time interval, and let
\[
        t\in I\longmapsto H(t)
\]
be a family of symmetric operators on \(\Hilb\).  We assume, at this
stage, only that
\[
        \Hilb^\infty\subset D(H(t)),
        \qquad
        H(t)\Hilb^\infty\subset \Hilb
\]
for every \(t\in I\), and that the maps
\[
        t\longmapsto H(t)u
\]
are continuous for every \(u\in \Hilb^\infty\).  Stronger assumptions
will be introduced later in order to prove convergence of the cut-off
dynamics.

For each \(N\), we define the cut-off Hamiltonian
\[
        H_N(t)=P_NH(t)P_N
\]
as an operator on \(\Hilb_N\).  Since \(\Hilb_N\) is finite-dimensional,
\(H_N(t)\) is a bounded operator on \(\Hilb_N\).  If \(H(t)\) is
symmetric, then \(H_N(t)\) is self-adjoint on \(\Hilb_N\).  Therefore
the finite-dimensional non-autonomous Schr\"odinger equation
\[
        \ii\partial_t u_N(t)=H_N(t)u_N(t),
        \qquad
        u_N(s)=u_{N,s}\in \Hilb_N,
\]
has a unique unitary propagator
\[
        U_N(t,s):\Hilb_N\to\Hilb_N.
\]
It is characterized by
\[
        U_N(s,s)=\Id_{\Hilb_N},
        \qquad
        U_N(t,r)U_N(r,s)=U_N(t,s),
\]
and
\[
        \ii\partial_t U_N(t,s)u=H_N(t)U_N(t,s)u.
\]

For an initial datum \(u_0\in\Hilb\), the cut-off solution is
\[
        u_N(t)=U_N(t,0)P_Nu_0.
\]
The main convergence problem, treated in the next sections, is to find
conditions under which
\[
        u_N(t)\longrightarrow u(t)
\]
strongly in \(\Hilb\), uniformly for \(t\) in compact intervals, where
\(u(t)\) solves the original Hamiltonian equation
\[
        \ii\partial_t u(t)=H(t)u(t),
        \qquad
        u(0)=u_0.
\]

\subsection{Time-sliced products}

Fix \(N\) and let
\[
        s=t_0<t_1<\cdots<t_M=t
\]
be a partition of the interval \([s,t]\).  We write
\[
        \Delta t_j=t_{j+1}-t_j.
\]
The time-sliced approximation of \(U_N(t,s)\) is defined by
\[
        U_{N,M}(t,s)
        =
        \exp\big(-\ii\Delta t_{M-1}H_N(t_{M-1})\big)
        \cdots
        \exp\big(-\ii\Delta t_0H_N(t_0)\big).
\]
Equivalently,
\[
        U_{N,M}(t,s)
        =
        \prod_{j=M-1}^{0}
        \exp\big(-\ii\Delta t_jH_N(t_j)\big),
\]
where the product is ordered from right to left in increasing time.

Since \(H_N(t)\) is a continuous family of matrices, the standard
finite-dimensional theory of non-autonomous linear systems gives
\[
        U_{N,M}(t,s)\longrightarrow U_N(t,s)
\]
as the mesh of the partition tends to zero:
\[
        |\Pi_M|
        :=
        \max_{0\leq j\leq M-1}\Delta t_j
        \longrightarrow0.
\]
Thus, for fixed \(N\), the propagator \(U_N(t,s)\) is already obtained
as the limit of elementary oscillatory factors.

\subsection{Finite-dimensional state-sum representation}
\label{subsec:finite-dimensional-state-sum}

The time-sliced product introduced above admits a canonical finite-dimensional
representation as a sum over intermediate states. This representation should
be distinguished from a configuration-space oscillatory integral, which
requires additional geometric structure.

Let
\[
d_{N}=\dim\mathcal{H}_{N},
\]
and choose an orthonormal basis
\[
\mathcal{B}_{N}=(e_{N,1},\ldots,e_{N,d_{N}})
\]
of \(\mathcal{H}_{N}\). For a partition
\[
\Pi_{M}=\{s=t_{0}<t_{1}<\cdots<t_{M}=t\},
\]
set
\[
E_{N,j}
=
\exp\bigl(-\mathrm{i}\Delta t_{j}H_{N}(t_{j})\bigr),
\qquad
0\leq j\leq M-1.
\]
Thus
\[
U_{N,M}(t,s)
=
E_{N,M-1}\cdots E_{N,0}.
\]

For \(a,b\in\{1,\ldots,d_{N}\}\), write
\[
E_{N,j}(a,b)
=
\langle e_{N,a},E_{N,j}e_{N,b}\rangle.
\]
Repeated insertion of the identity
\[
I_{\mathcal{H}_{N}}
=
\sum_{a=1}^{d_{N}}
|e_{N,a}\rangle\langle e_{N,a}|
\]
gives the following exact formula:
\[
\begin{split}
&\langle e_{N,a_{M}},U_{N,M}(t,s)e_{N,a_{0}}\rangle
\\
&\quad=
\sum_{a_{1},\ldots,a_{M-1}=1}^{d_{N}}
\prod_{j=0}^{M-1}
E_{N,j}(a_{j+1},a_{j}).
\end{split}
\]

\begin{definition}[Spectral cut-off state-sum amplitude]
\label{def:spectral-cutoff-state-sum}
For fixed \(N\), \(M\), and endpoints \(a_{0},a_{M}\), the quantity
\[
\mathcal{A}_{N,M}(a_{M},a_{0};t,s)
=
\sum_{a_{1},\ldots,a_{M-1}=1}^{d_{N}}
\prod_{j=0}^{M-1}
E_{N,j}(a_{j+1},a_{j})
\]
is called the spectral cut-off state-sum amplitude associated with the basis
\(\mathcal{B}_{N}\) and the partition \(\Pi_{M}\).
\end{definition}

By construction,
\[
\mathcal{A}_{N,M}(a_{M},a_{0};t,s)
=
\langle e_{N,a_{M}},U_{N,M}(t,s)e_{N,a_{0}}\rangle.
\]
The state-sum is therefore an exact representation of the time-sliced
propagator and not an additional analytic object.

\begin{remark}
\label{rem:basis-dependence-state-sum}
The intermediate state-sum depends on the choice of the orthonormal basis
\(\mathcal{B}_{N}\), whereas the resulting operator \(U_{N,M}(t,s)\) does not.
In an abstract finite-dimensional Hilbert space, insertion of an orthonormal
basis produces a finite sum, not an integral over
\(\mathbb{R}^{d_{N}}\). A continuous configuration-space integral requires a
realization of \(\mathcal{H}_{N}\) as a space of functions, or a continuous
resolution of the identity such as a coherent-state representation.
\end{remark}

For \(u\in\mathcal{H}\), one obtains
\[
\begin{split}
U_{N,M}(t,s)P_{N}u
={}&
\sum_{a_{M}=1}^{d_{N}}
\sum_{a_{0}=1}^{d_{N}}
\mathcal{A}_{N,M}(a_{M},a_{0};t,s)
\\
&\qquad\qquad\times
\langle e_{N,a_{0}},u\rangle e_{N,a_{M}}.
\end{split}
\]
Consequently,
\[
\lim_{|\Pi_{M}|\to 0}
\mathcal{A}_{N,M}(a,b;t,s)
=
\langle e_{N,a},U_{N}(t,s)e_{N,b}\rangle
\]
for each fixed \(N\) and \(a,b\).

\subsection{Exact configuration-space kernel formula}
\label{subsec:configuration-space-kernel}

A genuine integral representation is available when the ambient Hilbert space
has a configuration-space realization. Let \((X,\mu)\) be a
\(\sigma\)-finite measure space and assume that
\[
\mathcal{H}=L^{2}(X,\mu).
\]
Suppose that the range of \(P_{N}\) is spanned by an orthonormal family
\[
\varphi_{N,1},\ldots,\varphi_{N,d_{N}}.
\]
The integral kernel of \(P_{N}\) is then
\[
\Pi_{N}(x,y)
=
\sum_{a=1}^{d_{N}}
\varphi_{N,a}(x)\overline{\varphi_{N,a}(y)}.
\]

For each time slice, the finite-rank operator
\[
E_{N,j}
=
\exp\bigl(-\mathrm{i}\Delta t_{j}H_{N}(t_{j})\bigr)
\]
has the kernel
\[
K_{N,j}(x,y)
=
\sum_{a,b=1}^{d_{N}}
E_{N,j}(a,b)
\varphi_{N,a}(x)\overline{\varphi_{N,b}(y)}.
\]

\begin{proposition}[Exact time-sliced kernel formula]
\label{prop:exact-time-sliced-kernel}
Let \(u\in\mathcal{H}_{N}\). Then
\[
\begin{split}
&(U_{N,M}(t,s)u)(x_{M})
\\
&\quad=
\int_{X^{M}}
\left(
\prod_{j=0}^{M-1}
K_{N,j}(x_{j+1},x_{j})
\right)
u(x_{0})\,
d\mu(x_{0})\cdots d\mu(x_{M-1})
\end{split}
\]
for almost every \(x_{M}\in X\).
\end{proposition}

\begin{proof}
Each \(E_{N,j}\) is a finite-rank integral operator with kernel \(K_{N,j}\).
The formula follows by repeated composition of the kernels. Since all the
operators involved have finite-dimensional range, the multiple integral is
well defined in the \(L^{2}\)-sense. Equivalently, expanding every kernel in
the orthonormal basis reduces the formula to the finite state-sum of
Definition~\ref{def:spectral-cutoff-state-sum}.
\end{proof}

\begin{remark}
\label{rem:dimension-configuration-space}
The integration variables \(x_{j}\) belong to the configuration space \(X\).
Their dimension, when \(X\) is a manifold, is \(\dim X\) and not
\(d_{N}=\dim\mathcal{H}_{N}\). The spectral cut-off changes the rank of the
kernels but does not change the underlying configuration space.
\end{remark}

\subsection{Oscillatory realizations}
\label{subsec:oscillatory-realizations}

The kernel formula of
Proposition~\ref{prop:exact-time-sliced-kernel} is always available in a
configuration-space realization, but it is not automatically a classical
oscillatory integral. Such a representation requires additional information
on the one-step kernels.

\begin{definition}[Oscillatory realization of a one-step kernel]
\label{def:oscillatory-one-step-kernel}
An oscillatory realization of \(K_{N,j}\) is a representation
\[
K_{N,j}(x,y)
=
\operatorname{Os}
\int_{\mathbb{R}^{q_{N,j}}}
\exp\bigl(\mathrm{i}\Phi_{N,j}(x,y,\theta)\bigr)
a_{N,j}(x,y,\theta)\,d\theta,
\]
where \(\Phi_{N,j}\) is a real-valued phase function and
\(a_{N,j}\) is an amplitude belonging to a specified symbol class.
The notation \(\operatorname{Os}\int\) denotes the corresponding
regularized oscillatory integral.
\end{definition}

This definition includes the case \(q_{N,j}=0\), in which the kernel is
already given by
\[
K_{N,j}(x,y)
=
a_{N,j}(x,y)
\exp\bigl(\mathrm{i}\Phi_{N,j}(x,y)\bigr).
\]
However, the existence of such a factorization is not asserted in the
abstract Hilbert-space setting.

\begin{proposition}[Composition of oscillatory one-step kernels]
\label{prop:composition-oscillatory-kernels}
Assume that every one-step kernel \(K_{N,j}\) admits an oscillatory
realization as in
Definition~\ref{def:oscillatory-one-step-kernel}, and that the amplitudes and
phases satisfy conditions allowing the corresponding oscillatory integrals
to be composed. Then the kernel of \(U_{N,M}(t,s)\) admits the representation
\[
\begin{split}
K_{N,M}(x_{M},x_{0})
={}&
\operatorname{Os}
\int_{X^{M-1}}
\int_{\mathbb{R}^{q_{N,0}+\cdots+q_{N,M-1}}}
\\
&\exp\left(
\mathrm{i}
\sum_{j=0}^{M-1}
\Phi_{N,j}(x_{j+1},x_{j},\theta_{j})
\right)
\\
&\times
\prod_{j=0}^{M-1}
a_{N,j}(x_{j+1},x_{j},\theta_{j})
\\
&\times
d\theta_{0}\cdots d\theta_{M-1}
\,d\mu(x_{1})\cdots d\mu(x_{M-1}).
\end{split}
\]
\end{proposition}

\begin{proof}
This follows by inserting the oscillatory representation of each one-step
kernel into the exact composition formula of
Proposition~\ref{prop:exact-time-sliced-kernel}. The phase of the composed
integral is the sum of the one-step phases, and its amplitude is the product
of the one-step amplitudes. The stated compatibility assumptions justify the
successive regularized integrations.
\end{proof}

\begin{definition}[Spectral cut-off oscillatory amplitude]
\label{def:spectral-cutoff-oscillatory-amplitude}
Under the hypotheses of
Proposition~\ref{prop:composition-oscillatory-kernels}, the resulting
oscillatory representation of \(U_{N,M}(t,s)\) is called a spectral cut-off
oscillatory amplitude and is denoted by
\[
I_{N,M}(t,s).
\]
This notation refers to the oscillatory representation, while
\(U_{N,M}(t,s)\) denotes the underlying time-sliced operator.
\end{definition}

Thus,
\[
I_{N,M}(t,s)=U_{N,M}(t,s)
\]
as operators whenever the oscillatory realization exists. The notation
records additional information about the representation, not a different
operator.

\begin{remark}[Absence of a canonical action]
\label{rem:no-canonical-action}
In a general abstract Hilbert space, there is no canonical configuration
space, phase function, or classical action associated with the matrix
\(H_{N}(t)\). The finite state-sum remains well defined, but it should not be
identified with a Feynman-type integral. A discrete action arises only after
choosing additional geometric data, for example a Schrödinger
representation, a phase-space quantization, or a coherent-state resolution
of the identity.
\end{remark}

\subsection{The two limiting procedures}
\label{subsec:two-limiting-procedures}

For \(u\in\mathcal{H}\), define
\[
u_{N,M}(t)
=
U_{N,M}(t,0)P_{N}u.
\]
For each fixed \(N\), finite-dimensional non-autonomous evolution theory gives
\[
\lim_{|\Pi_{M}|\to 0}
u_{N,M}(t)
=
U_{N}(t,0)P_{N}u.
\]
If the hypotheses of the spectral convergence theorem are satisfied, then
\[
\lim_{N\to\infty}
U_{N}(t,0)P_{N}u
=
U(t,0)u.
\]
Therefore
\[
U(t,0)u
=
\lim_{N\to\infty}
\lim_{|\Pi_{M}|\to 0}
U_{N,M}(t,0)P_{N}u.
\]

At the abstract level, this formula is a double limit of finite-dimensional
time-sliced operators, equivalently of the state-sums introduced in
Definition~\ref{def:spectral-cutoff-state-sum}. When the one-step kernels
possess oscillatory realizations, the same formula becomes
\[
U(t,0)u
=
\lim_{N\to\infty}
\lim_{|\Pi_{M}|\to 0}
I_{N,M}(t,0)P_{N}u.
\]

\begin{remark}
\label{rem:iterated-versus-joint-limit}
The preceding formula concerns an iterated limit: the time-slicing limit is
taken first at fixed spectral cut-off, and the spectral cut-off is then
removed. A joint limiting procedure requires estimates on the time-slicing
error that are quantitative in \(N\). Such estimates will be established
below under additional time-regularity assumptions on \(H(t)\).
\end{remark}
\section{Convergence of the spectral cut-off dynamics}
\label{sec:convergence-cutoff-dynamics}

In the previous section we introduced the finite-dimensional
oscillatory amplitudes \(I_{N,M}(t,s)\) associated with the spectral
cut-off Hamiltonians
\[
        H_N(t)=P_NH(t)P_N.
\]
For fixed \(N\), the time-slicing limit gives the finite-dimensional
propagator \(U_N(t,s)\).  The purpose of this section is to prove that,
under suitable \(H_0\)-relative assumptions, the cut-off propagators
converge strongly to the propagator of the original non-autonomous
Hamiltonian equation.

The main point is that the convergence is not proved at the level of
formal path integrals.  It is proved at the level of propagators, using
the spectral cut-off and Duhamel's formula.  The oscillatory
interpretation is then inherited from the finite-dimensional
approximations.

\subsection{The Hilbert scale associated with \(H_0\)}

Let \(H_0\) be the positive self-adjoint reference operator introduced
in Section~\ref{sec:cutoff-oscillatory-integrals}.  We assume throughout
this section that \(H_0\) has compact resolvent.  For \(s\geq0\), we set
\[
        \Hilb^s=D((1+H_0)^s),
        \qquad
        \|u\|_s=\|(1+H_0)^s u\|.
\]
Thus \(\Hilb^0=\Hilb\).  We also write
\[
        \Hilb^\infty=\bigcap_{s\geq0}\Hilb^s.
\]

The spectral projections
\[
        P_N=\mathbf 1_{[0,N]}(H_0)
\]
commute with \(H_0\).  Hence, for every \(s\geq0\),
\[
        \|P_Nu\|_s\leq \|u\|_s,
        \qquad u\in\Hilb^s.
\]
Moreover,
\[
        P_Nu\longrightarrow u
        \qquad\text{in }\Hilb^s
\]
for every \(u\in\Hilb^s\).

We shall use the following elementary compactness lemma several times.

\begin{lemma}[Uniform convergence on compact sets]
\label{lem:uniform-tail}
Let \(s\geq0\), and let \(K\subset\Hilb^s\) be compact.  Then
\[
        \sup_{u\in K}\|(I-P_N)u\|_s
        \longrightarrow0
\]
as \(N\to+\infty\).
\end{lemma}

\begin{proof}
For each fixed \(u\in\Hilb^s\), we know that
\[
        \|(I-P_N)u\|_s\longrightarrow0.
\]
Moreover, \(P_N\) is a contraction on \(\Hilb^s\).  Hence
\[
        \|(I-P_N)u\|_s\leq 2\|u\|_s.
\]
Let \(\varepsilon>0\).  Since \(K\) is compact, there exist
\(u_1,\ldots,u_m\in K\) such that for every \(u\in K\) one can find
\(u_j\) with
\[
        \|u-u_j\|_s<\varepsilon.
\]
For this \(u_j\), we have
\[
\begin{aligned}
        \|(I-P_N)u\|_s
        &\leq
        \|(I-P_N)(u-u_j)\|_s
        +
        \|(I-P_N)u_j\|_s    \\
        &\leq
        2\|u-u_j\|_s
        +
        \|(I-P_N)u_j\|_s    \\
        &<
        2\varepsilon
        +
        \|(I-P_N)u_j\|_s .
\end{aligned}
\]
Since the set \(\{u_1,\ldots,u_m\}\) is finite, there exists
\(N_0\) such that
\[
        \|(I-P_N)u_j\|_s<\varepsilon
        \qquad
        \text{for all }j=1,\ldots,m
\]
whenever \(N\geq N_0\).  Therefore
\[
        \sup_{u\in K}\|(I-P_N)u\|_s
        \leq 3\varepsilon
\]
for \(N\geq N_0\).  This proves the claim.
\end{proof}

\subsection{Assumptions on the full dynamics}

We now introduce the assumptions under which the spectral cut-off
dynamics converges.

\begin{assumption}[Well-posed Hamiltonian evolution]
\label{ass:full-dynamics}
Let \(I\subset\mathbb R\) be a compact interval.  We assume that
\(H(t)\), \(t\in I\), is a family of symmetric operators on \(\Hilb\)
such that the non-autonomous Schr\"odinger equation
\[
        \ii\partial_t u(t)=H(t)u(t)
\]
generates a unitary propagator
\[
        U(t,s):\Hilb\to\Hilb,
        \qquad s,t\in I.
\]
Thus
\[
        U(t,r)U(r,s)=U(t,s),
        \qquad
        U(s,s)=\Id,
\]
and \(U(t,s)\) is strongly continuous in \((t,s)\).

We also assume that there exists a number \(\mu\geq0\) such that:

\begin{enumerate}[label=\textup{(\roman*)}]
\item
For every \(t\in I\), the operator \(H(t)\) extends to a bounded
operator
\[
        H(t):\Hilb^\mu\longrightarrow \Hilb.
\]

\item
The map
\[
        t\longmapsto H(t)
\]
is continuous from \(I\) to
\[
        \mathcal L(\Hilb^\mu,\Hilb).
\]

\item
The propagator preserves \(\Hilb^\mu\), and for every \(u\in\Hilb^\mu\)
the map
\[
        (t,s)\longmapsto U(t,s)u
\]
is continuous from \(I\times I\) to \(\Hilb^\mu\).

\item
For every \(u\in\Hilb^\mu\), the function
\[
        t\longmapsto U(t,s)u
\]
is a strong solution of
\[
        \ii\partial_t u(t)=H(t)u(t)
\]
with values in \(\Hilb\).
\end{enumerate}
\end{assumption}

\begin{remark}
The assumption above is not meant to be optimal.  It isolates the
minimal mechanism needed for the convergence proof: the full propagator
must be defined, it must preserve enough regularity to make \(H(t)u(t)\)
meaningful, and the high spectral tail of \(U(t,s)u\) must vanish in the
\(\Hilb^\mu\)-norm.
\end{remark}

\subsection{Finite-dimensional cut-off propagators}

For each \(N\), we define
\[
        \Hilb_N=P_N\Hilb,
        \qquad
        H_N(t)=P_NH(t)P_N.
\]
Since \(\Hilb_N\) is finite-dimensional, \(H_N(t)\) is a bounded
operator on \(\Hilb_N\).

\begin{lemma}[Finite-dimensional unitary propagator]
\label{lem:finite-propagator}
For each \(N\), the cut-off equation
\[
        \ii\partial_t u_N(t)=H_N(t)u_N(t),
        \qquad
        u_N(s)=u_{N,s}\in\Hilb_N,
\]
has a unique unitary propagator
\[
        U_N(t,s):\Hilb_N\to\Hilb_N.
\]
Moreover,
\[
        U_N(t,r)U_N(r,s)=U_N(t,s),
        \qquad
        U_N(s,s)=\Id_{\Hilb_N}.
\]
\end{lemma}

\begin{proof}
Since \(H(t):\Hilb^\mu\to\Hilb\) is norm-continuous and
\(P_N\Hilb\subset \Hilb^\infty\subset \Hilb^\mu\), the map
\[
        t\longmapsto H_N(t)=P_NH(t)P_N
\]
is a continuous map from \(I\) to
\(\mathcal L(\Hilb_N,\Hilb_N)\).  Hence the finite-dimensional linear
ordinary differential equation
\[
        \partial_t u_N(t)=-\ii H_N(t)u_N(t)
\]
has a unique global solution on \(I\).

It remains to prove unitarity.  Since \(H(t)\) is symmetric, for
\(v,w\in\Hilb_N\) we have
\[
\begin{aligned}
        \langle H_N(t)v,w\rangle
        &=
        \langle P_NH(t)P_Nv,w\rangle        \\
        &=
        \langle H(t)v,w\rangle              \\
        &=
        \langle v,H(t)w\rangle              \\
        &=
        \langle v,P_NH(t)P_Nw\rangle        \\
        &=
        \langle v,H_N(t)w\rangle.
\end{aligned}
\]
Thus \(H_N(t)\) is self-adjoint on the finite-dimensional Hilbert space
\(\Hilb_N\).  If \(u_N(t)\) and \(v_N(t)\) are two solutions, then
\[
\begin{aligned}
        \frac{d}{dt}\langle u_N(t),v_N(t)\rangle
        &=
        \langle -\ii H_N(t)u_N(t),v_N(t)\rangle
        +
        \langle u_N(t),-\ii H_N(t)v_N(t)\rangle' .
\end{aligned}
\]
Using the convention that the scalar product is linear in the first
variable, the second term is
\[
        \langle u_N(t),-\ii H_N(t)v_N(t)\rangle'
        =
        \ii\langle u_N(t),H_N(t)v_N(t)\rangle.
\]
Therefore
\[
        \frac{d}{dt}\langle u_N(t),v_N(t)\rangle
        =
        -\ii\langle H_N(t)u_N(t),v_N(t)\rangle
        +
        \ii\langle u_N(t),H_N(t)v_N(t)\rangle
        =
        0.
\]
The propagator is thus unitary.  The composition law follows from
uniqueness of solutions.
\end{proof}

\begin{remark}
If the convention for the Hilbert scalar product is conjugate-linear in
the first variable, the signs in the computation above are reversed,
but the conclusion is of course the same.
\end{remark}

\subsection{The Duhamel comparison formula}

We now compare the exact propagator \(U(t,s)\) with the cut-off
propagator \(U_N(t,s)\).

Let \(u\in\Hilb^\mu\).  We define
\[
        w_N(t)=P_NU(t,s)u,
        \qquad
        v_N(t)=U_N(t,s)P_Nu.
\]
Both functions take values in \(\Hilb_N\).  The function \(v_N\)
satisfies
\[
        \ii\partial_t v_N(t)=H_N(t)v_N(t),
        \qquad
        v_N(s)=P_Nu.
\]
The projected exact solution \(w_N\) satisfies a perturbed cut-off
equation.

\begin{lemma}[Projected exact dynamics]
\label{lem:projected-exact}
For \(u\in\Hilb^\mu\), one has
\[
        \ii\partial_t w_N(t)
        =
        H_N(t)w_N(t)
        +
        R_N(t,s)u,
\]
where
\[
        R_N(t,s)u
        =
        P_NH(t)(I-P_N)U(t,s)u.
\]
\end{lemma}

\begin{proof}
Since \(U(t,s)u\) is a strong solution,
\[
        \ii\partial_t U(t,s)u=H(t)U(t,s)u.
\]
Applying \(P_N\), we get
\[
        \ii\partial_t P_NU(t,s)u
        =
        P_NH(t)U(t,s)u.
\]
Now decompose
\[
        U(t,s)u=P_NU(t,s)u+(I-P_N)U(t,s)u.
\]
Then
\[
\begin{aligned}
        P_NH(t)U(t,s)u
        &=
        P_NH(t)P_NU(t,s)u
        +
        P_NH(t)(I-P_N)U(t,s)u     \\
        &=
        H_N(t)w_N(t)
        +
        R_N(t,s)u.
\end{aligned}
\]
This proves the formula.
\end{proof}

The difference \(w_N-v_N\) is therefore controlled by the remainder
\(R_N(t,s)u\).

\begin{proposition}[Duhamel formula for the cut-off error]
\label{prop:duhamel-cutoff}
For every \(u\in\Hilb^\mu\) and every \(s,t\in I\),
\[
        P_NU(t,s)u-U_N(t,s)P_Nu
        =
        -\ii
        \int_s^t
        U_N(t,\tau)
        P_NH(\tau)(I-P_N)U(\tau,s)u
        \,d\tau .
\]
Consequently,
\[
        \|P_NU(t,s)u-U_N(t,s)P_Nu\|
        \leq
        \int_s^t
        \|H(\tau)(I-P_N)U(\tau,s)u\|
        \,d\tau
\]
when \(t\geq s\), and the analogous estimate holds with the integral
over \([t,s]\) when \(t<s\).
\end{proposition}

\begin{proof}
Set
\[
        e_N(t)=w_N(t)-v_N(t).
\]
By Lemma~\ref{lem:projected-exact} and by the equation satisfied by
\(v_N\),
\[
        \ii\partial_t e_N(t)
        =
        H_N(t)e_N(t)+R_N(t,s)u.
\]
Moreover,
\[
        e_N(s)=P_Nu-P_Nu=0.
\]
The variation-of-constants formula in the finite-dimensional space
\(\Hilb_N\) gives
\[
        e_N(t)
        =
        -\ii
        \int_s^t
        U_N(t,\tau)R_N(\tau,s)u
        \,d\tau .
\]
Substituting the definition of \(R_N\), we obtain
\[
        e_N(t)
        =
        -\ii
        \int_s^t
        U_N(t,\tau)
        P_NH(\tau)(I-P_N)U(\tau,s)u
        \,d\tau.
\]
This proves the identity.

Since \(U_N(t,\tau)\) is unitary on \(\Hilb_N\) and \(P_N\) is an
orthogonal projection on \(\Hilb\),
\[
\begin{aligned}
        \|e_N(t)\|
        &\leq
        \int_s^t
        \|U_N(t,\tau)
        P_NH(\tau)(I-P_N)U(\tau,s)u\|
        \,d\tau                                    \\
        &\leq
        \int_s^t
        \|H(\tau)(I-P_N)U(\tau,s)u\|
        \,d\tau .
\end{aligned}
\]
This gives the estimate for \(t\geq s\).  The case \(t<s\) is identical,
with the orientation of the integral reversed.
\end{proof}

\subsection{Strong convergence of the cut-off propagators}

We now prove the main theorem of this section.

\begin{theorem}[Strong convergence of the spectral cut-off dynamics]
\label{thm:strong-cutoff-convergence}
Assume Assumption~\ref{ass:full-dynamics}.  Then, for every
\(u\in\Hilb^\mu\),
\[
        U_N(t,s)P_Nu
        \longrightarrow
        U(t,s)u
\]
strongly in \(\Hilb\), uniformly for \((t,s)\in I\times I\).
\end{theorem}

\begin{proof}
We write
\[
\begin{aligned}
        U_N(t,s)P_Nu-U(t,s)u
        &=
        \bigl(U_N(t,s)P_Nu-P_NU(t,s)u\bigr)
        +
        \bigl(P_NU(t,s)u-U(t,s)u\bigr).
\end{aligned}
\]
We estimate the two terms separately.

First consider the projection tail
\[
        (P_N-I)U(t,s)u.
\]
By Assumption~\ref{ass:full-dynamics}, the map
\[
        (t,s)\longmapsto U(t,s)u
\]
is continuous from the compact set \(I\times I\) to \(\Hilb^\mu\).
Hence
\[
        K_u=\{U(t,s)u;(t,s)\in I\times I\}
\]
is compact in \(\Hilb^\mu\).  By Lemma~\ref{lem:uniform-tail},
\[
        \sup_{(t,s)\in I\times I}
        \|(I-P_N)U(t,s)u\|_\mu
        \longrightarrow0.
\]
In particular,
\[
        \sup_{(t,s)\in I\times I}
        \|(I-P_N)U(t,s)u\|
        \longrightarrow0.
\]

We now estimate the first term.  By
Proposition~\ref{prop:duhamel-cutoff},
\[
\begin{aligned}
        \|P_NU(t,s)u-U_N(t,s)P_Nu\|
        &\leq
        \int_{\min(s,t)}^{\max(s,t)}
        \|H(\tau)(I-P_N)U(\tau,s)u\|
        \,d\tau .
\end{aligned}
\]
Since \(t\mapsto H(t)\) is continuous from \(I\) to
\(\mathcal L(\Hilb^\mu,\Hilb)\), there exists a constant \(C_I>0\) such
that
\[
        \|H(\tau)z\|
        \leq C_I\|z\|_\mu
\]
for all \(\tau\in I\) and all \(z\in\Hilb^\mu\).  Therefore
\[
\begin{aligned}
        \|P_NU(t,s)u-U_N(t,s)P_Nu\|
        &\leq
        C_I |t-s|
        \sup_{\tau\in I}
        \|(I-P_N)U(\tau,s)u\|_\mu .
\end{aligned}
\]
Again, the set
\[
        \{U(\tau,s)u;(\tau,s)\in I\times I\}
\]
is compact in \(\Hilb^\mu\).  Hence, by
Lemma~\ref{lem:uniform-tail},
\[
        \sup_{(\tau,s)\in I\times I}
        \|(I-P_N)U(\tau,s)u\|_\mu
        \longrightarrow0.
\]
It follows that
\[
        \sup_{(t,s)\in I\times I}
        \|P_NU(t,s)u-U_N(t,s)P_Nu\|
        \longrightarrow0.
\]
Combining this with the projection-tail estimate proves
\[
        \sup_{(t,s)\in I\times I}
        \|U_N(t,s)P_Nu-U(t,s)u\|
        \longrightarrow0.
\]
This is the desired uniform strong convergence.
\end{proof}

\begin{remark}
The proof shows precisely where the spectral cut-off enters.  The only
error term is
\[
        P_NH(t)(I-P_N)U(t,s)u.
\]
Thus the convergence follows from the fact that the high-energy tail of
the exact solution is small in the regularity norm required to apply
\(H(t)\).
\end{remark}

\subsection{Consequences for state-sum and oscillatory amplitudes}

Recall that for fixed \(N\), the time-sliced operator \(U_{N,M}(t,s)\)
converges to \(U_N(t,s)\) as the mesh of the partition tends to zero.
Combining this finite-dimensional convergence with
Theorem~\ref{thm:strong-cutoff-convergence}, we obtain the following
consequence.

\begin{corollary}[State-sum construction of the solution]
\label{cor:state-sum-construction}
Let Assumption~\ref{ass:full-dynamics} hold.  Let \(u\in\Hilb^\mu\).
For each \(N\) and each time partition, consider the time-sliced operator
\[
        U_{N,M}(t,s)
        =
        \prod_{j=M-1}^{0}
        \exp\bigl(-\ii\Delta t_jH_N(t_j)\bigr).
\]
Then
\[
        U(t,s)u
        =
        \lim_{N\to\infty}
        \lim_{M\to\infty}
        U_{N,M}(t,s)P_Nu
\]
strongly in \(\Hilb\), uniformly for \((t,s)\in I\times I\).
Equivalently, the finite-dimensional state-sum representation of
\(U_{N,M}(t,s)P_Nu\) converges to \(U(t,s)u\) in this iterated sense.
\end{corollary}

\begin{proof}
For fixed \(N\), the finite-dimensional time-slicing construction gives
\[
        \lim_{M\to\infty}
        U_{N,M}(t,s)P_Nu
        =
        U_N(t,s)P_Nu
\]
in \(\Hilb_N\).  By Theorem~\ref{thm:strong-cutoff-convergence},
\[
        \lim_{N\to\infty}
        U_N(t,s)P_Nu
        =
        U(t,s)u
\]
strongly in \(\Hilb\), uniformly for \((t,s)\in I\times I\).  Combining
the two limits gives the result.
\end{proof}

\begin{corollary}[Conditional oscillatory construction]
\label{cor:oscillatory-construction}
Under the assumptions of
Corollary~\ref{cor:state-sum-construction}, suppose in addition that the
one-step kernels admit oscillatory realizations in the sense of
Definition~\ref{def:oscillatory-one-step-kernel}. Then
\[
 U(t,s)u
 =
 \lim_{N\to\infty}\lim_{M\to\infty}
 I_{N,M}(t,s)P_Nu
\]
strongly in \(\Hilb\), uniformly for \((t,s)\in I\times I\).
\end{corollary}

\begin{proof}
By definition of an oscillatory realization,
\[
 I_{N,M}(t,s)=U_{N,M}(t,s)
\]
as operators. The conclusion therefore follows from
Corollary~\ref{cor:state-sum-construction}.
\end{proof}

\begin{remark}
The abstract construction is intrinsically a limit of finite-dimensional
state-sums. It becomes a limit of oscillatory integrals only in realizations
for which the required phases and amplitudes have been constructed.
\end{remark}

\subsection{Quantitative spectral convergence}
\label{subsec:quantitative-spectral-convergence}

The preceding convergence theorem is qualitative. We now show that additional
regularity of the exact evolution gives an explicit convergence rate with
respect to the spectral cut-off. Set
\[
\Lambda=1+H_{0}.
\]
Since
\[
P_{N}=\mathbf{1}_{[0,N]}(H_{0})
      =\mathbf{1}_{[1,1+N]}(\Lambda),
\]
the spectral theorem gives the following elementary tail estimate.

\begin{lemma}[Quantitative spectral tail]
\label{lem:quantitative-spectral-tail}
Let \(r\geq 0\) and \(\sigma\geq 0\). Then, for every
\(u\in\Hilb^{r+\sigma}\),
\[
\|(I-P_{N})u\|_{r}
\leq
(1+N)^{-\sigma}\|u\|_{r+\sigma}.
\]
\end{lemma}

\begin{proof}
By the spectral theorem,
\[
\|(I-P_{N})u\|_{r}^{2}
=
\int_{(N,\infty)}
(1+\lambda)^{2r}\,d\langle E_{H_{0}}(\lambda)u,u\rangle.
\]
For \(\lambda>N\), one has
\[
(1+\lambda)^{2r}
\leq
(1+N)^{-2\sigma}(1+\lambda)^{2(r+\sigma)}.
\]
Consequently,
\[
\|(I-P_{N})u\|_{r}^{2}
\leq
(1+N)^{-2\sigma}\|u\|_{r+\sigma}^{2},
\]
which proves the claim.
\end{proof}

We can now quantify the convergence established in
Theorem~\ref{thm:strong-cutoff-convergence}.

\begin{theorem}[Quantitative convergence of the cut-off dynamics]
\label{thm:quantitative-cutoff-convergence}
Let \(I\subset\mathbb{R}\) be compact, let \(\mu\geq 0\), and assume that
\[
H(t)\in\mathcal{L}(\Hilb^{\mu},\Hilb)
\]
depends continuously on \(t\in I\). Suppose that the non-autonomous equation
generates a unitary propagator \(U(t,s)\) and that, for some \(\sigma>0\),
\[
U(t,s)\Hilb^{\mu+\sigma}
\subset
\Hilb^{\mu+\sigma}.
\]
Assume moreover that
\[
C_{U,\mu+\sigma}
:=
\sup_{s,t\in I}
\|U(t,s)\|_{\mathcal{L}(\Hilb^{\mu+\sigma})}
<\infty.
\]
Set
\[
C_{H,\mu}
:=
\sup_{t\in I}
\|H(t)\|_{\mathcal{L}(\Hilb^{\mu},\Hilb)}.
\]
Then, for every \(u\in\Hilb^{\mu+\sigma}\),
\[
\begin{split}
\sup_{s,t\in I}
\|U_{N}(t,s)P_{N}u-U(t,s)u\|
\leq{}&
C_{U,\mu+\sigma}
\Bigl(
(1+N)^{-(\mu+\sigma)}
\\
&\quad
+
|I|C_{H,\mu}(1+N)^{-\sigma}
\Bigr)
\|u\|_{\mu+\sigma}.
\end{split}
\]
In particular, there exists a constant \(C_{I,\mu,\sigma}>0\),
independent of \(N\), such that
\[
\sup_{s,t\in I}
\|U_{N}(t,s)P_{N}u-U(t,s)u\|
\leq
C_{I,\mu,\sigma}(1+N)^{-\sigma}
\|u\|_{\mu+\sigma}.
\]
\end{theorem}

\begin{proof}
As in the proof of Theorem~\ref{thm:strong-cutoff-convergence}, write
\[
\begin{split}
U_{N}(t,s)P_{N}u-U(t,s)u
={}&
U_{N}(t,s)P_{N}u-P_{N}U(t,s)u
\\
&+
(P_{N}-I)U(t,s)u.
\end{split}
\]
The second term is estimated using
Lemma~\ref{lem:quantitative-spectral-tail} with \(r=0\):
\[
\begin{split}
\|(I-P_{N})U(t,s)u\|
&\leq
(1+N)^{-(\mu+\sigma)}
\|U(t,s)u\|_{\mu+\sigma}
\\
&\leq
C_{U,\mu+\sigma}
(1+N)^{-(\mu+\sigma)}
\|u\|_{\mu+\sigma}.
\end{split}
\]

For the first term, Proposition~\ref{prop:duhamel-cutoff} gives
\[
\begin{split}
&\|P_{N}U(t,s)u-U_{N}(t,s)P_{N}u\|
\\
&\qquad\leq
\int_{\min(s,t)}^{\max(s,t)}
\|H(\tau)(I-P_{N})U(\tau,s)u\|\,d\tau
\\
&\qquad\leq
C_{H,\mu}|t-s|
\sup_{\tau\in I}
\|(I-P_{N})U(\tau,s)u\|_{\mu}.
\end{split}
\]
Applying Lemma~\ref{lem:quantitative-spectral-tail} with \(r=\mu\)
yields
\[
\|(I-P_{N})U(\tau,s)u\|_{\mu}
\leq
(1+N)^{-\sigma}
\|U(\tau,s)u\|_{\mu+\sigma}.
\]
Therefore,
\[
\begin{split}
&\|P_{N}U(t,s)u-U_{N}(t,s)P_{N}u\|
\\
&\qquad\leq
|I|C_{H,\mu}C_{U,\mu+\sigma}
(1+N)^{-\sigma}
\|u\|_{\mu+\sigma}.
\end{split}
\]
Combining the two estimates proves the result.
\end{proof}

\begin{remark}
\label{rem:spectral-rate-interpretation}
The convergence rate is determined by the regularity of the exact orbit
relative to the reference operator \(H_{0}\). In concrete elliptic settings,
the parameter \(\sigma\) measures the number of additional Sobolev derivatives
of the initial datum and of the propagated solution. The estimate is therefore
the natural spectral analogue of the usual rate of Galerkin approximation.
\end{remark}

\subsection{Uniform stability from commutator estimates}
\label{subsec:commutator-stability}

The quantitative theorem above still uses regularity bounds for the exact
propagator. We next give an intrinsic sufficient condition for the uniform
stability of the cut-off propagators. This condition is formulated directly
in terms of the commutator of \(H(t)\) with the Hilbert scale generated by
\(\Lambda=1+H_{0}\).

\begin{assumption}[Energy commutator bound]
\label{ass:energy-commutator-bound}
Let \(r\geq 0\). Assume that there exists a locally integrable function
\(c_{r}:I\to[0,\infty)\) such that
\[
\left|
\ii
\left(
\langle H(t)v,\Lambda^{2r}v\rangle
-
\langle \Lambda^{2r}v,H(t)v\rangle
\right)
\right|
\leq
c_{r}(t)\|v\|_{r}^{2}
\]
for every \(v\in\Hilb^{\infty}\) and almost every \(t\in I\).
\end{assumption}

Equivalently, whenever the commutator is defined as a quadratic form,
Assumption~\ref{ass:energy-commutator-bound} reads
\[
\left|
\langle v,\ii[H(t),\Lambda^{2r}]v\rangle
\right|
\leq
c_{r}(t)\|v\|_{r}^{2}.
\]

\begin{proposition}[Uniform stability of the cut-off propagators]
\label{prop:uniform-cutoff-stability}
Assume Assumption~\ref{ass:energy-commutator-bound}. Then, for every \(N\),
every \(v\in\Hilb_{N}\), and every \(s,t\in I\),
\[
\|U_{N}(t,s)v\|_{r}
\leq
\exp\left(
\frac{1}{2}
\int_{\min(s,t)}^{\max(s,t)}c_{r}(\tau)\,d\tau
\right)
\|v\|_{r}.
\]
In particular, the stability constant is independent of \(N\).
\end{proposition}

\begin{proof}
Fix \(N\), \(s\in I\), and \(v\in\Hilb_{N}\), and set
\[
v_{N}(t)=U_{N}(t,s)v.
\]
Since \(\Hilb_{N}\subset\Hilb^{\infty}\), all the following
computations take place in a finite-dimensional space. We have
\[
\ii\partial_{t}v_{N}(t)=H_{N}(t)v_{N}(t).
\]
Hence
\[
\begin{split}
\frac{d}{dt}\|v_{N}(t)\|_{r}^{2}
&=
\frac{d}{dt}
\langle \Lambda^{r}v_{N}(t),\Lambda^{r}v_{N}(t)\rangle
\\
&=
\ii
\left(
\langle H_{N}(t)v_{N}(t),\Lambda^{2r}v_{N}(t)\rangle
\right.
\\
&\hspace{35mm}\left.
-
\langle \Lambda^{2r}v_{N}(t),H_{N}(t)v_{N}(t)\rangle
\right).
\end{split}
\]
Because \(P_{N}\) commutes with \(\Lambda\) and
\(v_{N}(t)\in\Hilb_{N}\), the expression on the right-hand side equals
\[
\ii
\left(
\langle H(t)v_{N}(t),\Lambda^{2r}v_{N}(t)\rangle
-
\langle \Lambda^{2r}v_{N}(t),H(t)v_{N}(t)\rangle
\right).
\]
Assumption~\ref{ass:energy-commutator-bound} therefore implies
\[
\left|
\frac{d}{dt}\|v_{N}(t)\|_{r}^{2}
\right|
\leq
c_{r}(t)\|v_{N}(t)\|_{r}^{2}.
\]
Gronwall's inequality gives
\[
\|v_{N}(t)\|_{r}^{2}
\leq
\exp\left(
\int_{\min(s,t)}^{\max(s,t)}c_{r}(\tau)\,d\tau
\right)
\|v\|_{r}^{2}.
\]
Taking square roots proves the result.
\end{proof}

\begin{remark}
The important point is that the estimate is obtained before passing to the
limit and is uniform in the spectral cut-off. Thus the stability assumption
used in the energy-norm convergence theorem is no longer merely postulated:
it follows from a commutator estimate on the original Hamiltonian family.
\end{remark}

\subsection{Construction of the propagator by removal of the cut-off}
\label{subsec:constructive-propagator}

We now remove the assumption that the full propagator is known in advance.
Under uniform estimates on the spectral scale, the cut-off propagators form a
Cauchy family and therefore construct the non-autonomous evolution.

\begin{theorem}[Construction by spectral cut-off]
\label{thm:construction-by-spectral-cutoff}
Let \(I\subset\mathbb{R}\) be compact. Assume that there exist
\(\mu\geq 0\) and \(\sigma>0\) such that the following conditions hold.

\begin{enumerate}
\item For every \(r\geq 0\), the family \(H(t)\) is norm-continuous as a map
\[
I\longrightarrow
\mathcal{L}(\Hilb^{r+\mu},\Hilb^{r}).
\]

\item The energy commutator bound of
Assumption~\ref{ass:energy-commutator-bound} holds for
\(r=\mu+\sigma\), with
\[
c_{\mu+\sigma}\in L^{1}(I).
\]

\item Each \(H(t)\) is symmetric on the common core
\(\Hilb^{\infty}\).
\end{enumerate}

Then, for every \(u\in\Hilb^{\mu+\sigma}\), the family
\[
U_{N}(t,s)P_{N}u
\]
is Cauchy in \(\Hilb\), uniformly for \((t,s)\in I\times I\).
Consequently, the limit
\[
U(t,s)u
:=
\lim_{N\to\infty}U_{N}(t,s)P_{N}u
\]
exists uniformly on \(I\times I\).

Moreover, there exists a constant \(C_{I,\mu,\sigma}>0\) such that, whenever
\(M\geq N\),
\[
\sup_{s,t\in I}
\|U_{M}(t,s)P_{M}u-U_{N}(t,s)P_{N}u\|
\leq
C_{I,\mu,\sigma}(1+N)^{-\sigma}
\|u\|_{\mu+\sigma}.
\]
The limiting family extends uniquely to a strongly continuous unitary
propagator on \(\Hilb\).

If the preceding assumptions hold at every level of the Hilbert scale, then
\(U(t,s)\) preserves \(\Hilb^{\infty}\), and for
\(u\in\Hilb^{\infty}\) one has
\[
\ii\partial_{t}U(t,s)u
=
H(t)U(t,s)u.
\]
Thus the full non-autonomous Hamiltonian evolution is constructed as the
spectral cut-off limit.
\end{theorem}

\begin{proof}
Fix \(M\geq N\), \(s\in I\), and
\(u\in\Hilb^{\mu+\sigma}\). Set
\[
w_{M}(t)=U_{M}(t,s)P_{M}u
\]
and
\[
z_{N,M}(t)=P_{N}w_{M}(t).
\]
Since \(P_{N}P_{M}=P_{N}\), we have
\[
z_{N,M}(s)=P_{N}u.
\]
Moreover,
\[
\begin{split}
\ii\partial_{t}z_{N,M}(t)
&=
P_{N}H(t)P_{M}w_{M}(t)
\\
&=
P_{N}H(t)P_{N}z_{N,M}(t)
+
P_{N}H(t)(P_{M}-P_{N})w_{M}(t)
\\
&=
H_{N}(t)z_{N,M}(t)
+
P_{N}H(t)(P_{M}-P_{N})w_{M}(t).
\end{split}
\]
Variation of constants in \(\Hilb_{N}\) gives
\[
\begin{split}
z_{N,M}(t)-U_{N}(t,s)P_{N}u
=
-\ii
\int_{s}^{t}
U_{N}(t,\tau)
P_{N}H(\tau)(P_{M}-P_{N})w_{M}(\tau)\,d\tau.
\end{split}
\]
By unitarity of \(U_{N}\) in \(\Hilb\),
\[
\begin{split}
&\|z_{N,M}(t)-U_{N}(t,s)P_{N}u\|
\\
&\qquad\leq
C_{H,\mu}|I|
\sup_{\tau\in I}
\|(P_{M}-P_{N})w_{M}(\tau)\|_{\mu},
\end{split}
\]
where
\[
C_{H,\mu}
=
\sup_{\tau\in I}
\|H(\tau)\|_{\mathcal{L}(\Hilb^{\mu},\Hilb)}.
\]
Since
\[
P_{M}-P_{N}=P_{M}(I-P_{N}),
\]
Lemma~\ref{lem:quantitative-spectral-tail} implies
\[
\|(P_{M}-P_{N})w_{M}(\tau)\|_{\mu}
\leq
(1+N)^{-\sigma}\|w_{M}(\tau)\|_{\mu+\sigma}.
\]
By Proposition~\ref{prop:uniform-cutoff-stability},
\[
\sup_{\tau,s\in I}
\|U_{M}(\tau,s)P_{M}u\|_{\mu+\sigma}
\leq
C_{\mathrm{stab}}\|u\|_{\mu+\sigma},
\]
where
\[
C_{\mathrm{stab}}
=
\exp\left(
\frac{1}{2}\int_{I}c_{\mu+\sigma}(\tau)\,d\tau
\right)
\]
is independent of \(M\). Hence
\[
\sup_{s,t\in I}
\|z_{N,M}(t)-U_{N}(t,s)P_{N}u\|
\leq
C_{H,\mu}|I|C_{\mathrm{stab}}
(1+N)^{-\sigma}
\|u\|_{\mu+\sigma}.
\]

It remains to estimate the component of \(w_{M}(t)\) orthogonal to
\(\Hilb_{N}\). By Lemma~\ref{lem:quantitative-spectral-tail},
\[
\begin{split}
\|(I-P_{N})w_{M}(t)\|
&\leq
(1+N)^{-(\mu+\sigma)}
\|w_{M}(t)\|_{\mu+\sigma}
\\
&\leq
C_{\mathrm{stab}}
(1+N)^{-(\mu+\sigma)}
\|u\|_{\mu+\sigma}.
\end{split}
\]
Since
\[
w_{M}(t)-U_{N}(t,s)P_{N}u
=
(I-P_{N})w_{M}(t)
+
z_{N,M}(t)-U_{N}(t,s)P_{N}u,
\]
we obtain the asserted Cauchy estimate.

The limit therefore exists for
\(u\in\Hilb^{\mu+\sigma}\). Since this space is dense in
\(\Hilb\) and
\[
\|U_{N}(t,s)P_{N}u\|=\|P_{N}u\|\leq\|u\|,
\]
the limiting operators extend uniquely to contractions on \(\Hilb\).
Applying the same construction to the backward propagators \(U_{N}(s,t)\)
shows that
\[
U(t,s)U(s,t)=U(s,t)U(t,s)=I.
\]
Thus \(U(t,s)\) is unitary.

The propagator identity follows by passing to the limit in
\[
U_{N}(t,r)U_{N}(r,s)=U_{N}(t,s),
\]
first on the dense regularity space and then on \(\Hilb\).
Strong continuity follows from the uniform convergence on \(I\times I\).

Finally, if the assumptions hold at every level of the Hilbert scale, the
same argument gives convergence in each \(\Hilb^{r}\). For
\(u\in\Hilb^{\infty}\), one may therefore pass to the limit in the
integral equation
\[
U_{N}(t,s)P_{N}u
=
P_{N}u
-
\ii
\int_{s}^{t}
P_{N}H(\tau)P_{N}
U_{N}(\tau,s)P_{N}u\,d\tau.
\]
This yields
\[
U(t,s)u
=
u
-
\ii
\int_{s}^{t}H(\tau)U(\tau,s)u\,d\tau.
\]
The norm-continuity of
\[
t\longmapsto H(t)\in
\mathcal{L}(\Hilb^{r+\mu},\Hilb^{r})
\]
then implies that \(t\mapsto U(t,s)u\) is differentiable in
\(\Hilb\) and satisfies
\[
\ii\partial_{t}U(t,s)u=H(t)U(t,s)u.
\]
\end{proof}
\begin{corollary}[Quantitative convergence in energy norms]
\label{cor:quantitative-energy-convergence}
Let \(r\geq0\) and \(\sigma>0\). Assume that
\[
H(t):\Hilb^{r+\mu}\longrightarrow\Hilb^r
\]
is norm-continuous in \(t\), and that the energy commutator bound holds
at both levels \(r\) and \(r+\mu+\sigma\). Then there exists
\(C_{I,r,\mu,\sigma}>0\) such that
\[
\sup_{s,t\in I}
\|U_N(t,s)P_Nu-U(t,s)u\|_r
\leq
C_{I,r,\mu,\sigma}(1+N)^{-\sigma}
\|u\|_{r+\mu+\sigma}
\]
for every \(u\in\Hilb^{r+\mu+\sigma}\).
\end{corollary}

\begin{proof}
The proof of Theorem~\ref{thm:construction-by-spectral-cutoff} is
repeated in the Hilbert space \(\Hilb^r\).
Proposition~\ref{prop:uniform-cutoff-stability}, applied at level \(r\),
controls the cut-off propagator in the Duhamel term, while the same
proposition at level \(r+\mu+\sigma\) controls the high-regularity norm
of the comparison dynamics.  Lemma~
\ref{lem:quantitative-spectral-tail} gives
\[
\|(I-P_N)v\|_{r+\mu}
\leq
(1+N)^{-\sigma}\|v\|_{r+\mu+\sigma}.
\]
The resulting Cauchy estimate is uniform for
\((t,s)\in I\times I\). Passing to the limit gives the stated
inequality.
\end{proof}

\begin{proof}
The proof of Theorem~\ref{thm:construction-by-spectral-cutoff} is
repeated in the Hilbert space \(\Hilb^r\).  Proposition~
\ref{prop:uniform-cutoff-stability}, applied at level \(r\), replaces
unitarity in the estimate for the Duhamel term.  Lemma~
\ref{lem:quantitative-spectral-tail} gives
\[
\|(I-P_N)v\|_{r+\mu}
\leq
(1+N)^{-\sigma}\|v\|_{r+\mu+\sigma}.
\]
The resulting Cauchy estimate is uniform for \((t,s)\in I\times I\).
Passing to the limit gives the stated inequality.
\end{proof}
\begin{remark}[Removal of the logical circularity]
\label{rem:removal-circularity}
Theorem~\ref{thm:construction-by-spectral-cutoff} reverses the logical order
of Theorem~\ref{thm:strong-cutoff-convergence}. The full propagator is no longer an
input of the approximation theorem. It is obtained as the uniform strong
limit of the finite-dimensional cut-off propagators. Consequently, the
spectral approximation becomes a genuine construction of the evolution
rather than a consistency statement relative to an evolution assumed to
exist in advance.
\end{remark}

\subsection{Quantitative time-slicing estimates}
\label{subsec:quantitative-time-slicing}

We now estimate the time-slicing error uniformly with respect to the spectral
cut-off. Assume that there exist \(\alpha\in(0,1]\) and
\(L_{\alpha,\mu}>0\) such that
\[
 \|H(t)-H(s)\|_{\mathcal L(\Hilb^\mu,\Hilb)}
 \leq L_{\alpha,\mu}|t-s|^\alpha,
 \qquad s,t\in I.
\]
For a partition
\[
 \Pi=\{s=t_0<t_1<\cdots<t_M=t\},
 \qquad
 |\Pi|=\max_j(t_{j+1}-t_j),
\]
define
\[
 U_{N,\Pi}(t,s)
 =
 \prod_{j=M-1}^{0}
 \exp\bigl(-\ii(t_{j+1}-t_j)H_N(t_j)\bigr).
\]

\begin{lemma}[Hölder estimate for the cut-off Hamiltonians]
\label{lem:cutoff-holder-estimate}
For every \(N\geq0\) and every \(s,t\in I\),
\[
 \|H_N(t)-H_N(s)\|_{\mathcal L(\Hilb_N)}
 \leq
 L_{\alpha,\mu}(1+N)^\mu|t-s|^\alpha.
\]
\end{lemma}

\begin{proof}
For \(v\in\Hilb_N\),
\[
\begin{aligned}
 \|(H_N(t)-H_N(s))v\|
 &\leq L_{\alpha,\mu}|t-s|^\alpha\|v\|_\mu \\
 &\leq L_{\alpha,\mu}(1+N)^\mu
 |t-s|^\alpha\|v\|,
\end{aligned}
\]
because the spectral support of \(v\) is contained in \([0,N]\).
\end{proof}

\begin{proposition}[Time-slicing error at spectral level \(N\)]
\label{prop:uniform-time-slicing-error}
The exact cut-off propagator and its left-endpoint time-sliced approximation
satisfy
\[
 \|U_{N,\Pi}(t,s)-U_N(t,s)\|_{\mathcal L(\Hilb_N)}
 \leq
 \frac{L_{\alpha,\mu}}{\alpha+1}
 (1+N)^\mu |t-s| |\Pi|^\alpha.
\]
\end{proposition}

\begin{proof}
Let
\[
 H_{N,\Pi}(\tau)=H_N(t_j),
 \qquad \tau\in[t_j,t_{j+1}).
\]
Its propagator is \(U_{N,\Pi}\). Duhamel's formula and unitarity give
\[
\begin{aligned}
 \|U_N(t,s)-U_{N,\Pi}(t,s)\|
 &\leq
 \int_s^t\|H_N(\tau)-H_{N,\Pi}(\tau)\|\,d\tau \\
 &\leq
 \frac{L_{\alpha,\mu}}{\alpha+1}(1+N)^\mu
 \sum_{j=0}^{M-1}(t_{j+1}-t_j)^{\alpha+1} \\
 &\leq
 \frac{L_{\alpha,\mu}}{\alpha+1}(1+N)^\mu
 |t-s||\Pi|^\alpha.
\end{aligned}
\]
For \(t<s\), the same argument is applied on \([t,s]\).
\end{proof}

\begin{remark}
The factor \((1+N)^\mu\) reflects the growth of the cut-off Hamiltonian in
the ambient Hilbert norm. Thus a joint limit requires the mesh of the time
partition to decrease sufficiently rapidly relative to the spectral cut-off.
\end{remark}

\subsection{Joint removal of the spectral and temporal cut-offs}
\label{subsec:joint-removal-cutoffs}

Combining the spectral convergence rate with the preceding time-slicing
estimate gives a quantitative joint approximation theorem.

\begin{theorem}[Joint spectral and time-slicing estimate]
\label{thm:joint-spectral-time-slicing}
Assume the hypotheses of
Theorem~\ref{thm:construction-by-spectral-cutoff}. Suppose moreover that
\(H\) is Hölder continuous of exponent \(\alpha\in(0,1]\) as a map
\[
 I\longrightarrow\mathcal L(\Hilb^\mu,\Hilb),
\]
and that the uniform energy estimate used in the constructive theorem holds
at level \(\mu+\sigma\), for some \(\sigma>0\). Then there exists a constant
\(C_{I,\mu,\sigma,\alpha}>0\), independent of \(N\) and \(\Pi\), such that
\[
\begin{aligned}
 &\|U_{N,\Pi}(t,s)P_Nu-U(t,s)u\| \\
 &\qquad\leq
 C_{I,\mu,\sigma,\alpha}
 \left((1+N)^{-\sigma}+(1+N)^\mu|\Pi|^\alpha\right)
 \|u\|_{\mu+\sigma}
\end{aligned}
\]
for every \(u\in\Hilb^{\mu+\sigma}\) and every \(s,t\in I\).

Consequently, if
\[
 N_k\longrightarrow\infty,
 \qquad
 |\Pi_k|\longrightarrow0,
 \qquad
 (1+N_k)^\mu|\Pi_k|^\alpha\longrightarrow0,
\]
then
\[
 U_{N_k,\Pi_k}(t,s)P_{N_k}u
 \longrightarrow U(t,s)u
\]
in \(\Hilb\), uniformly for \((t,s)\in I\times I\).
\end{theorem}

\begin{proof}
Decompose
\[
\begin{aligned}
 U_{N,\Pi}(t,s)P_Nu-U(t,s)u
 ={}&
 (U_{N,\Pi}(t,s)-U_N(t,s))P_Nu \\
 &+U_N(t,s)P_Nu-U(t,s)u.
\end{aligned}
\]
Proposition~\ref{prop:uniform-time-slicing-error} gives
\[
 \|(U_{N,\Pi}(t,s)-U_N(t,s))P_Nu\|
 \leq
 \frac{L_{\alpha,\mu}|I|}{\alpha+1}
 (1+N)^\mu|\Pi|^\alpha\|u\|.
\]
The quantitative spectral estimate obtained in the proof of
Theorem~\ref{thm:construction-by-spectral-cutoff} gives
\[
 \sup_{s,t\in I}
 \|U_N(t,s)P_Nu-U(t,s)u\|
 \leq
 C_{I,\mu,\sigma}(1+N)^{-\sigma}\|u\|_{\mu+\sigma}.
\]
Combining the two inequalities proves the estimate and the convergence
statement.
\end{proof}

\begin{corollary}[Uniform time partitions]
\label{cor:uniform-time-partitions}
For a uniform partition of \([s,t]\) into \(M\) subintervals,
\[
\begin{aligned}
 &\|U_{N,M}(t,s)P_Nu-U(t,s)u\| \\
 &\qquad\leq
 C_{I,\mu,\sigma,\alpha}
 \left((1+N)^{-\sigma}+(1+N)^\mu M^{-\alpha}\right)
 \|u\|_{\mu+\sigma}.
\end{aligned}
\]
In particular, if
\[
 \frac{M(N)}{(1+N)^{\mu/\alpha}}\longrightarrow+\infty,
\]
then
\[
 U_{N,M(N)}(t,s)P_Nu\longrightarrow U(t,s)u
\]
uniformly for \((t,s)\in I\times I\).
\end{corollary}

\begin{corollary}[Balanced choice of the time-slicing parameter]
\label{cor:balanced-time-slicing}
If
\[
 M(N)=
 \left\lceil(1+N)^{(\mu+\sigma)/\alpha}\right\rceil,
\]
then
\[
 \sup_{s,t\in I}
 \|U_{N,M(N)}(t,s)P_Nu-U(t,s)u\|
 \leq
 C_{I,\mu,\sigma,\alpha}(1+N)^{-\sigma}
 \|u\|_{\mu+\sigma}.
\]
\end{corollary}

\begin{proof}
The choice of \(M(N)\) gives
\[
 (1+N)^\mu M(N)^{-\alpha}
 \leq (1+N)^{-\sigma}.
\]
The conclusion follows from
Corollary~\ref{cor:uniform-time-partitions}.
\end{proof}

\subsection{Joint convergence of state-sum and oscillatory amplitudes}
\label{subsec:joint-convergence-amplitudes}

The preceding theorem applies directly to the finite-dimensional state-sum
representation introduced in
Definition~\ref{def:spectral-cutoff-state-sum}.

\begin{corollary}[Joint convergence of the state-sum construction]
\label{cor:joint-state-sum-convergence}
Let \(u\in\Hilb^{\mu+\sigma}\), and let \(M(N)\) satisfy
\[
 (1+N)^\mu M(N)^{-\alpha}\longrightarrow0.
\]
Then the state-sum representation of
\(U_{N,M(N)}(t,s)P_Nu\) converges in \(\Hilb\), uniformly for
\((t,s)\in I\times I\), to \(U(t,s)u\).
\end{corollary}

When the one-step kernels possess oscillatory realizations in the sense of
Definition~\ref{def:oscillatory-one-step-kernel}, the same conclusion holds
for the corresponding oscillatory amplitudes.

\begin{corollary}[Joint convergence of oscillatory realizations]
\label{cor:joint-oscillatory-convergence}
Assume that each time-sliced cut-off propagator admits an oscillatory
realization
\[
 I_{N,M}(t,s)=U_{N,M}(t,s)
\]
as an operator. If
\[
 (1+N)^\mu M(N)^{-\alpha}\longrightarrow0,
\]
then
\[
 I_{N,M(N)}(t,s)P_Nu
 \longrightarrow U(t,s)u
\]
in \(\Hilb\), uniformly for \((t,s)\in I\times I\).
For the balanced choice
\[
 M(N)=\left\lceil(1+N)^{(\mu+\sigma)/\alpha}\right\rceil,
\]
one has
\[
 \sup_{s,t\in I}
 \|I_{N,M(N)}(t,s)P_Nu-U(t,s)u\|
 \leq
 C_{I,\mu,\sigma,\alpha}(1+N)^{-\sigma}
 \|u\|_{\mu+\sigma}.
\]
\end{corollary}

\begin{remark}
The joint approximation replaces the iterated prescription
\[
 \lim_{N\to\infty}\lim_{M\to\infty}
 I_{N,M}(t,s)P_Nu
\]
by the genuine one-parameter limit
\[
 \lim_{N\to\infty}I_{N,M(N)}(t,s)P_Nu.
\]
The admissible growth of \(M(N)\) is determined explicitly by the derivative
loss \(\mu\), the temporal Hölder exponent \(\alpha\), and the additional
regularity \(\sigma\).
\end{remark}

\section{Periodic cut-offs and effective Hamiltonians}
\label{sec:periodic-effective-hamiltonians}

We now assume that the Hamiltonian is periodic in time.  The purpose of
this section is to show how the spectral cut-off construction interacts
with finite-dimensional Floquet theory and with the Floquet--Magnus
effective Hamiltonians.

The main point is the following.  For each fixed spectral cut-off \(N\),
the Hamiltonian
\[
        H_N(t)=P_NH(t)P_N
\]
acts on a finite-dimensional Hilbert space.  Therefore the standard
Floquet--Magnus expansion is available at this level.  One can then ask
whether the finite-dimensional effective Hamiltonians converge, as
\(N\to\infty\), to a formal effective Hamiltonian associated with the
unbounded dynamics.  We prove such a convergence at the level of
finite-order Floquet--Magnus coefficients on the \(H_0\)-smooth scale.

\subsection{Periodic rescaling}

Let \(H(t)\), \(t\in\mathbb R\), be a \(1\)-periodic family of symmetric
operators:
\[
        H(t+1)=H(t).
\]
For \(T>0\), we define the corresponding \(T\)-periodic Hamiltonian by
\[
        H^{(T)}(t)=H(t/T).
\]
The associated Schr\"odinger equation is
\[
        \ii \partial_t u(t)=H^{(T)}(t)u(t).
\]
The parameter \(T\) is the driving period.  The high-frequency regime is
the limit
\[
        T\to0^+.
\]

For each spectral cut-off \(N\), we define
\[
        H_N(t)=P_NH(t)P_N
\]
on
\[
        \Hilb_N=P_N\Hilb,
\]
and similarly
\[
        H_N^{(T)}(t)=P_NH^{(T)}(t)P_N
        =
        H_N(t/T).
\]
Let
\[
        U_N^{(T)}(t,s)
\]
be the finite-dimensional unitary propagator generated by
\(H_N^{(T)}(t)\).

Since \(H_N^{(T)}\) is \(T\)-periodic, the dynamics over one period is
encoded by the monodromy matrix
\[
        U_N^{(T)}(T,0).
\]

\subsection{Finite-dimensional Floquet Hamiltonians}

For fixed \(N\), Floquet theory applies to \(U_N^{(T)}(T,0)\).  Since
\(U_N^{(T)}(T,0)\) is a unitary matrix on \(\Hilb_N\), one may choose a
self-adjoint matrix \(H_{F,N}^{(T)}\) such that
\[
        U_N^{(T)}(T,0)
        =
        \ee^{-\ii T H_{F,N}^{(T)}}.
\]
The operator \(H_{F,N}^{(T)}\) is called a Floquet Hamiltonian.  It is
not unique, because the logarithm of a unitary matrix is not unique.

In the high-frequency regime, one seeks an asymptotic expansion of
\(H_{F,N}^{(T)}\) in powers of \(T\).  The Floquet--Magnus expansion
provides such an expansion:
\[
        H_{\mathrm{FM},N}^{(T)}
        \sim
        \sum_{\ell\geq0}T^\ell H_{\mathrm{FM},N}^{[\ell]}.
\]
The first two coefficients are
\[
        H_{\mathrm{FM},N}^{[0]}
        =
        \int_0^1 H_N(\tau)\dd\tau,
\]
and, with the convention used here,
\[
        H_{\mathrm{FM},N}^{[1]}
        =
        \frac{\ii}{2}
        \int_0^1
        \int_0^{\tau_1}
        [H_N(\tau_1),H_N(\tau_2)]
        \dd\tau_2\,\dd\tau_1.
\]
Higher coefficients are finite linear combinations of iterated
integrals of nested commutators of \(H_N(\tau)\) at different times.

For \(L\geq0\), we write
\[
        H_{\mathrm{FM},L,N}^{(T)}
        =
        \sum_{\ell=0}^{L}
        T^\ell H_{\mathrm{FM},N}^{[\ell]}.
\]
This is the cut-off Floquet--Magnus effective Hamiltonian of order
\(L\).

\begin{remark}
For fixed \(N\), all operators are bounded matrices.  Thus the
Floquet--Magnus coefficients are unambiguously defined.  However, the
operator norms of \(H_N(t)\) may grow with \(N\).  Consequently, the
convergence radius of the full Magnus series is not expected to be
uniform in \(N\).  In this paper we only use finite-order coefficients
and their convergence on smooth vectors.
\end{remark}

\subsection{Regularity assumptions for the coefficients}

We now introduce assumptions under which the cut-off
Floquet--Magnus coefficients converge as \(N\to\infty\).

Let \(\mu\geq0\).  We assume that, for every \(r\geq0\), the maps
\[
        H(t):\Hilb^{r+\mu}\longrightarrow \Hilb^r
\]
are bounded and depend continuously on \(t\in[0,1]\) in the operator
norm of
\[
        \mathcal L(\Hilb^{r+\mu},\Hilb^r).
\]
Equivalently, for every \(r\geq0\), there exists \(C_r>0\) such that
\[
        \|H(t)u\|_r
        \leq
        C_r\|u\|_{r+\mu},
        \qquad
        t\in[0,1].
\]
This means that \(H(t)\) loses at most \(\mu\) derivatives with respect
to the \(H_0\)-scale.

For a fixed order \(L\), the coefficient
\(H_{\mathrm{FM}}^{[\ell]}\) involves products of at most
\(\ell+1\) Hamiltonians.  Therefore it is natural to work on
\[
        \Hilb^{r+(L+1)\mu}.
\]

\subsection{Convergence of cut-off words}

We first prove a basic approximation lemma for products of Hamiltonians
with spectral projections inserted between the factors.

For \(m\geq1\), define the exact word
\[
        W_m(t_m,\ldots,t_1)
        =
        H(t_m)H(t_{m-1})\cdots H(t_1),
\]
initially as a bounded operator
\[
        W_m(t_m,\ldots,t_1):
        \Hilb^{r+m\mu}\longrightarrow \Hilb^r.
\]
Similarly, define the cut-off word
\[
        W_{m,N}(t_m,\ldots,t_1)
        =
        P_NH(t_m)P_NH(t_{m-1})P_N\cdots P_NH(t_1)P_N.
\]

\begin{lemma}[Convergence of cut-off words]
\label{lem:word-convergence}
Let \(m\geq0\), \(r\geq0\), and \(u\in \Hilb^{r+m\mu}\).  Then
\[
        W_{m,N}(t_m,\ldots,t_1)u
        \longrightarrow
        W_m(t_m,\ldots,t_1)u
\]
in \(\Hilb^r\), uniformly for
\[
        (t_1,\ldots,t_m)\in[0,1]^m.
\]
For \(m=0\), this statement means simply
\[
        P_Nu\to u
        \qquad\text{in }\Hilb^r.
\]
\end{lemma}

\begin{proof}
We prove the result by induction on \(m\).

For \(m=0\), the claim is precisely the strong convergence
\[
        P_Nu\to u
\]
in \(\Hilb^r\).

Assume that the result is true for words of length \(m-1\).  Let
\(u\in\Hilb^{r+m\mu}\).  We write
\[
        W_m(t_m,\ldots,t_1)
        =
        H(t_m)W_{m-1}(t_{m-1},\ldots,t_1),
\]
and
\[
        W_{m,N}(t_m,\ldots,t_1)
        =
        P_NH(t_m)P_NW_{m-1,N}(t_{m-1},\ldots,t_1).
\]
We estimate
\[
\begin{aligned}
& W_{m,N}(t_m,\ldots,t_1)u
  -
  W_m(t_m,\ldots,t_1)u                                      \\[1mm]
& =
  P_NH(t_m)P_N
  \Bigl(
        W_{m-1,N}(t_{m-1},\ldots,t_1)u
        -
        W_{m-1}(t_{m-1},\ldots,t_1)u
  \Bigr)                                                     \\
&\quad
  +
  P_NH(t_m)(P_N-I)
  W_{m-1}(t_{m-1},\ldots,t_1)u                               \\
&\quad
  +
  (P_N-I)H(t_m)
  W_{m-1}(t_{m-1},\ldots,t_1)u .
\end{aligned}
\]
We show that the three terms tend to zero in \(\Hilb^r\), uniformly in
the time variables.

For the first term, since \(P_N\) is a contraction on every
\(\Hilb^r\), and since
\[
        H(t_m):\Hilb^{r+\mu}\to\Hilb^r
\]
is uniformly bounded for \(t_m\in[0,1]\), there exists \(C_r>0\) such
that
\[
\begin{aligned}
&\left\|
P_NH(t_m)P_N
\Bigl(
W_{m-1,N}u-W_{m-1}u
\Bigr)
\right\|_r                                             \\
&\qquad\leq
C_r
\left\|
W_{m-1,N}(t_{m-1},\ldots,t_1)u
-
W_{m-1}(t_{m-1},\ldots,t_1)u
\right\|_{r+\mu}.
\end{aligned}
\]
The induction hypothesis applies in the space \(\Hilb^{r+\mu}\), since
\[
        u\in\Hilb^{r+m\mu}
        =
        \Hilb^{(r+\mu)+(m-1)\mu}.
\]
Thus the first term tends to zero uniformly.

For the second term, consider the set
\[
        K
        =
        \left\{
        W_{m-1}(t_{m-1},\ldots,t_1)u;
        (t_1,\ldots,t_{m-1})\in[0,1]^{m-1}
        \right\}.
\]
By the norm-continuity of the maps \(t\mapsto H(t)\), this set is
compact in \(\Hilb^{r+\mu}\).  By Lemma~\ref{lem:uniform-tail},
\[
        \sup_{v\in K}\|(P_N-I)v\|_{r+\mu}
        \longrightarrow0.
\]
Using again the uniform boundedness of
\[
        H(t_m):\Hilb^{r+\mu}\to\Hilb^r,
\]
we obtain
\[
        \sup_{t_1,\ldots,t_m}
        \left\|
        P_NH(t_m)(P_N-I)
        W_{m-1}(t_{m-1},\ldots,t_1)u
        \right\|_r
        \longrightarrow0.
\]

For the third term, the set
\[
        K'
        =
        \left\{
        H(t_m)W_{m-1}(t_{m-1},\ldots,t_1)u;
        (t_1,\ldots,t_m)\in[0,1]^m
        \right\}
\]
is compact in \(\Hilb^r\).  Again by Lemma~\ref{lem:uniform-tail},
\[
        \sup_{v\in K'}\|(P_N-I)v\|_r
        \longrightarrow0.
\]
Therefore the third term also tends to zero uniformly.

Combining the three estimates proves the induction step, and hence the
lemma.
\end{proof}

\subsection{Formal unbounded Floquet--Magnus coefficients}

We now define the formal coefficients associated with the unbounded
Hamiltonian \(H(t)\).  For each \(\ell\geq0\), the
\(\ell\)-th Floquet--Magnus coefficient is a universal Lie polynomial
in the family \(H(t)\), involving iterated integrals of nested
commutators of length \(\ell+1\).  We denote it by
\[
        H_{\mathrm{FM}}^{[\ell]}.
\]
Thus
\[
        H_{\mathrm{FM}}^{[0]}
        =
        \int_0^1 H(\tau)\dd\tau,
\]
and
\[
        H_{\mathrm{FM}}^{[1]}
        =
        \frac{\ii}{2}
        \int_0^1
        \int_0^{\tau_1}
        [H(\tau_1),H(\tau_2)]
        \dd\tau_2\,\dd\tau_1.
\]
The higher coefficients are defined similarly as iterated integrals of
nested commutators.

Since products of \(\ell+1\) Hamiltonians lose at most
\((\ell+1)\mu\) derivatives, the coefficient
\(H_{\mathrm{FM}}^{[\ell]}\) is well defined as a bounded operator
\[
        H_{\mathrm{FM}}^{[\ell]}:
        \Hilb^{r+(\ell+1)\mu}\longrightarrow \Hilb^r.
\]
For fixed \(L\), we define the formal effective Hamiltonian of order
\(L\) by
\[
        H_{\mathrm{FM},L}^{(T)}
        =
        \sum_{\ell=0}^{L}
        T^\ell H_{\mathrm{FM}}^{[\ell]}.
\]
It is therefore defined on
\[
        \Hilb^{r+(L+1)\mu}
\]
as an operator into \(\Hilb^r\).

\begin{remark}
At this stage, \(H_{\mathrm{FM},L}^{(T)}\) is only a formal finite-order
operator on the \(H_0\)-scale.  Questions of symmetry, self-adjointness,
and generation of a unitary group require additional hypotheses.  The
purpose of this section is only to establish convergence of the cut-off
coefficients toward these formal coefficients.
\end{remark}

\subsection{Convergence of the Floquet--Magnus coefficients}

We now prove that the finite-dimensional coefficients converge to the
formal unbounded coefficients.

\begin{proposition}[Coefficient convergence]
\label{prop:FM-coefficient-convergence}
Let \(\ell\geq0\), \(r\geq0\), and
\(u\in\Hilb^{r+(\ell+1)\mu}\).  Then
\[
        H_{\mathrm{FM},N}^{[\ell]}P_Nu
        \longrightarrow
        H_{\mathrm{FM}}^{[\ell]}u
\]
in \(\Hilb^r\) as \(N\to\infty\).
\end{proposition}

\begin{proof}
The coefficient \(H_{\mathrm{FM},N}^{[\ell]}\) is a finite linear
combination of iterated integrals of nested commutators involving
\(\ell+1\) factors
\[
        H_N(\tau_1),\ldots,H_N(\tau_{\ell+1}).
\]
Each nested commutator expands into a finite linear combination of
ordered products of these factors.  Therefore it is enough to prove the
claim for a term of the form
\[
        \int_{\Delta}
        H_N(\tau_{\sigma(\ell+1)})
        \cdots
        H_N(\tau_{\sigma(1)})
        P_Nu
        \,d\tau,
\]
where \(\Delta\subset[0,1]^{\ell+1}\) is a simplex and
\(\sigma\) is a permutation.

Since
\[
        H_N(\tau)=P_NH(\tau)P_N,
\]
the integrand is precisely a cut-off word of length \(\ell+1\).  By
Lemma~\ref{lem:word-convergence}, it converges in \(\Hilb^r\),
uniformly for \(\tau\in\Delta\), to
\[
        H(\tau_{\sigma(\ell+1)})
        \cdots
        H(\tau_{\sigma(1)})u.
\]
Uniform convergence allows us to pass to the limit under the integral.
Thus each word contribution converges to its unbounded counterpart.
Since only finitely many word contributions occur in the expansion of
the nested commutators, the whole coefficient converges:
\[
        H_{\mathrm{FM},N}^{[\ell]}P_Nu
        \longrightarrow
        H_{\mathrm{FM}}^{[\ell]}u.
\]
\end{proof}

\begin{corollary}[Convergence of finite-order effective Hamiltonians]
\label{cor:effective-H-convergence}
Let \(L\geq0\), \(r\geq0\), and
\(u\in\Hilb^{r+(L+1)\mu}\).  Then, for every fixed \(T>0\),
\[
        H_{\mathrm{FM},L,N}^{(T)}P_Nu
        \longrightarrow
        H_{\mathrm{FM},L}^{(T)}u
\]
in \(\Hilb^r\) as \(N\to\infty\).
\end{corollary}

\begin{proof}
By definition,
\[
        H_{\mathrm{FM},L,N}^{(T)}
        =
        \sum_{\ell=0}^{L}
        T^\ell H_{\mathrm{FM},N}^{[\ell]},
\]
and
\[
        H_{\mathrm{FM},L}^{(T)}
        =
        \sum_{\ell=0}^{L}
        T^\ell H_{\mathrm{FM}}^{[\ell]}.
\]
For each \(\ell\leq L\), Proposition~\ref{prop:FM-coefficient-convergence}
applies because
\[
        u\in\Hilb^{r+(L+1)\mu}
        \subset
        \Hilb^{r+(\ell+1)\mu}.
\]
Since the sum is finite, the result follows immediately.
\end{proof}

\subsection{Finite-dimensional error estimates at fixed cut-off}

We now recall the standard finite-dimensional meaning of the truncated
Floquet--Magnus expansion.  This result is included to clarify what is
obtained before the cut-off is removed.

For fixed \(N\), set
\[
        B_N=\sup_{\tau\in[0,1]}\|H_N(\tau)\|_{\mathcal L(\Hilb_N)}.
\]
Since \(\Hilb_N\) is finite-dimensional and \(H_N(\tau)\) is continuous,
\(B_N<+\infty\).

\begin{proposition}[Finite-dimensional Magnus estimate]
\label{prop:finite-Magnus-estimate}
Fix \(N\) and \(L\geq0\).  There exist \(T_N>0\) and \(C_{N,L}>0\) such
that, for all \(0<T<T_N\),
\[
        \left\|
        U_N^{(T)}(T,0)
        -
        \exp\left(-\ii T H_{\mathrm{FM},L,N}^{(T)}\right)
        \right\|
        \leq
        C_{N,L}T^{L+2}.
\]
Moreover, for every \(q\in\mathbb Z\),
\[
        \left\|
        U_N^{(T)}(qT,0)
        -
        \exp\left(-\ii qT H_{\mathrm{FM},L,N}^{(T)}\right)
        \right\|
        \leq
        |q|\,C_{N,L}T^{L+2}
\]
whenever \(0<T<T_N\).
\end{proposition}

\begin{proof}
For fixed \(N\), the Hamiltonian \(H_N(\tau)\) is a continuous
matrix-valued \(1\)-periodic function.  The rescaled Hamiltonian is
\[
        H_N^{(T)}(t)=H_N(t/T).
\]
Over one period,
\[
        \int_0^T \|H_N^{(T)}(t)\|\,dt
        =
        T\int_0^1 \|H_N(\tau)\|\,d\tau
        \leq
        TB_N.
\]
For \(T>0\) small enough, the Magnus series for the matrix equation
converges absolutely.  Thus there exists a matrix
\(\Omega_N^{(T)}(T)\) such that
\[
        U_N^{(T)}(T,0)
        =
        \exp\bigl(\Omega_N^{(T)}(T)\bigr),
\]
and
\[
        \Omega_N^{(T)}(T)
        =
        -\ii T
        \sum_{\ell=0}^{\infty}
        T^\ell H_{\mathrm{FM},N}^{[\ell]}.
\]
The truncation at order \(L\) satisfies
\[
        \Omega_N^{(T)}(T)
        =
        -\ii T H_{\mathrm{FM},L,N}^{(T)}
        +
        R_{N,L}^{(T)},
\]
with
\[
        \|R_{N,L}^{(T)}\|
        \leq
        C'_{N,L}T^{L+2}
\]
for \(T\) sufficiently small.  This follows from the absolute
convergence of the Magnus series in finite dimension.

We now compare the exponentials.  For bounded matrices \(A\) and \(B\),
\[
        \ee^A-\ee^B
        =
        \int_0^1
        \ee^{(1-\theta)A}(A-B)\ee^{\theta B}
        \,d\theta.
\]
Applying this with
\[
        A=\Omega_N^{(T)}(T),
        \qquad
        B=-\ii T H_{\mathrm{FM},L,N}^{(T)},
\]
and using the boundedness of the exponentials for \(T\) small, we get
\[
        \left\|
        \ee^A-\ee^B
        \right\|
        \leq
        C''_{N,L}\|A-B\|
        \leq
        C_{N,L}T^{L+2}.
\]
Hence
\[
        \left\|
        U_N^{(T)}(T,0)
        -
        \exp\left(-\ii T H_{\mathrm{FM},L,N}^{(T)}\right)
        \right\|
        \leq
        C_{N,L}T^{L+2}.
\]

For stroboscopic times \(qT\), periodicity gives
\[
        U_N^{(T)}(qT,0)
        =
        \left(U_N^{(T)}(T,0)\right)^q
\]
for \(q\geq0\).  Set
\[
        A_T=U_N^{(T)}(T,0),
        \qquad
        B_T=\exp\left(-\ii T H_{\mathrm{FM},L,N}^{(T)}\right).
\]
Both \(A_T\) and \(B_T\) are unitary matrices, because
\(H_{\mathrm{FM},L,N}^{(T)}\) is self-adjoint.  The telescopic identity
\[
        A_T^q-B_T^q
        =
        \sum_{j=0}^{q-1}
        A_T^{q-1-j}(A_T-B_T)B_T^j
\]
gives
\[
        \|A_T^q-B_T^q\|
        \leq
        q\|A_T-B_T\|
        \leq
        qC_{N,L}T^{L+2}.
\]
The case \(q<0\) follows by applying the same argument to the inverses.
This proves the estimate.
\end{proof}

\begin{remark}
The constants in Proposition~\ref{prop:finite-Magnus-estimate} depend
on \(N\).  This is unavoidable without additional uniform bounds on
\(\|H_N(t)\|\).  The proposition is therefore not a uniform
high-frequency theorem in \(N\).  Its role is to justify the effective
Hamiltonian at each finite cut-off level.
\end{remark}

\subsection{Compatibility diagram}

Combining the convergence results of the previous section with the
coefficient convergence proved above, one obtains the following
interpretation.

For fixed \(N\), the chain of constructions is
\[
        H_N^{(T)}(t)
        \longrightarrow
        U_N^{(T)}(t,s)
        \longrightarrow
        H_{\mathrm{FM},L,N}^{(T)}.
\]
The first arrow is the finite-dimensional Hamiltonian evolution.  The
second arrow is the finite-dimensional Floquet--Magnus construction.

As \(N\to\infty\), the propagators converge strongly:
\[
        U_N^{(T)}(t,s)P_Nu
        \longrightarrow
        U^{(T)}(t,s)u,
\]
under the assumptions of Theorem~\ref{thm:strong-cutoff-convergence}.
At the same time, the finite-order effective Hamiltonians converge on
smooth vectors:
\[
        H_{\mathrm{FM},L,N}^{(T)}P_Nu
        \longrightarrow
        H_{\mathrm{FM},L}^{(T)}u.
\]
This gives the commutative picture
\[
\begin{array}{ccc}
H_N^{(T)}(t)
& \longrightarrow &
U_N^{(T)}(t,s)
\\[1mm]
\downarrow \mathrm{FM}_L
&&
\downarrow N\to\infty
\\[1mm]
H_{\mathrm{FM},L,N}^{(T)}
& \longrightarrow &
U^{(T)}(t,s)
\end{array}
\]
where the lower-left object also converges, on the \(H_0\)-smooth scale,
to the formal unbounded effective Hamiltonian
\[
        H_{\mathrm{FM},L}^{(T)}.
\]
A fully commutative diagram at the level of effective propagators
requires an additional self-adjointness and resolvent convergence
statement for
\[
        H_{\mathrm{FM},L,N}^{(T)}
        \longrightarrow
        H_{\mathrm{FM},L}^{(T)}.
\]
This is a separate issue, which we isolate below.

\subsection{A conditional convergence result for effective propagators}

The finite-order operator
\[
        H_{\mathrm{FM},L}^{(T)}
\]
is generally unbounded.  Even if it is symmetric on
\(\Hilb^\infty\), it need not be essentially self-adjoint without
additional assumptions.  Therefore convergence of the coefficients does
not automatically imply convergence of the effective propagators.

We record the standard conditional statement.

\begin{proposition}[Effective propagators under strong resolvent convergence]
\label{prop:effective-propagator-convergence}
Fix \(T>0\) and \(L\geq0\).  Suppose that
\(H_{\mathrm{FM},L}^{(T)}\) admits a self-adjoint extension
\[
        \widehat H_{\mathrm{FM},L}^{(T)}.
\]
Extend each finite-dimensional operator
\(H_{\mathrm{FM},L,N}^{(T)}\) to a self-adjoint operator on \(\Hilb\) by
setting it equal to zero on \((P_N\Hilb)^\perp\).  Assume that
\[
        H_{\mathrm{FM},L,N}^{(T)}
        \longrightarrow
        \widehat H_{\mathrm{FM},L}^{(T)}
\]
in the strong resolvent sense.  Then, for every \(u\in\Hilb\) and every
\(t\in\mathbb R\),
\[
        \exp\left(-\ii tH_{\mathrm{FM},L,N}^{(T)}\right)u
        \longrightarrow
        \exp\left(-\ii t\widehat H_{\mathrm{FM},L}^{(T)}\right)u.
\]
The convergence is uniform for \(t\) in compact intervals.
\end{proposition}

\begin{proof}
This is the standard strong-resolvent convergence theorem for
self-adjoint operators.  Strong resolvent convergence of self-adjoint
operators is equivalent to strong convergence of the corresponding
unitary groups
\[
        \ee^{-\ii tA_N}
        \longrightarrow
        \ee^{-\ii tA}
\]
for each fixed \(t\), and the convergence is uniform for \(t\) in
compact intervals.  Applying this theorem with
\[
        A_N=H_{\mathrm{FM},L,N}^{(T)},
        \qquad
        A=\widehat H_{\mathrm{FM},L}^{(T)}
\]
gives the result.
\end{proof}

\begin{remark}
The proposition shows exactly what remains to be proved in concrete
models if one wants convergence not only of the effective Hamiltonian
coefficients but also of the effective autonomous dynamics.  One must
establish a self-adjoint realization of the limiting effective
Hamiltonian and prove strong resolvent convergence of the cut-off
effective Hamiltonians.
\end{remark}

\subsection{Conclusion of the periodic construction}

We have proved three facts.

First, at each finite spectral cut-off, the periodically driven system
has a standard finite-dimensional Floquet--Magnus expansion.

Second, the finite-order Floquet--Magnus coefficients converge on the
\(H_0\)-Hilbert scale toward formal unbounded coefficients:
\[
        H_{\mathrm{FM},N}^{[\ell]}P_Nu
        \longrightarrow
        H_{\mathrm{FM}}^{[\ell]}u.
\]

Third, the corresponding effective Hamiltonians satisfy
\[
        H_{\mathrm{FM},L,N}^{(T)}P_Nu
        \longrightarrow
        H_{\mathrm{FM},L}^{(T)}u
\]
for smooth enough \(u\).

Thus the spectral cut-off construction is compatible with finite-order
effective Hamiltonians.  The cut-off oscillatory integrals construct
the approximate dynamics; finite-dimensional Floquet theory extracts an
autonomous effective Hamiltonian at each cut-off; and the coefficients
of this effective Hamiltonian converge, on smooth vectors, to the
formal unbounded Floquet--Magnus coefficients.
\section{Examples and classes of Hamiltonians}
\label{sec:examples}

We now describe several classes of Hamiltonians to which the spectral
cut-off construction applies.  The purpose of this section is not to
give the most general possible assumptions in each setting, but rather
to show that the abstract hypotheses used above are natural in standard
functional analytic situations.

Throughout this section, the guiding principle is the following.  One
chooses a positive self-adjoint reference operator \(H_0\), usually
elliptic and with compact resolvent, and works with the Hilbert scale
\[
        \Hilb^s=D((1+H_0)^s).
\]
If the time-dependent Hamiltonians satisfy estimates of the form
\[
        H(t):\Hilb^{s+\mu}\longrightarrow \Hilb^s,
\]
uniformly for \(t\) in compact intervals, and if the corresponding
non-autonomous Schr\"odinger equation is well posed, then the spectral
cut-off dynamics constructed in the previous sections applies.

\subsection{Schrödinger operators with time-dependent potentials}
\label{subsec:schrodinger-time-dependent-potentials}

Let \(M\) be a compact smooth Riemannian manifold without boundary, and let
\(\Delta\geq 0\) denote the non-negative Laplace--Beltrami operator. We work
on
\[
\mathcal{H}=L^{2}(M)
\]
and choose the reference operator
\[
H_{0}=(1+\Delta)^{1/2}.
\]
Set
\[
\Lambda=1+H_{0}.
\]
The scale
\[
\mathcal{H}^{r}=D(\Lambda^{r})
\]
is equivalent to the usual Sobolev scale \(H^{r}(M)\).

Consider
\[
H(t)=\Delta+V(t),
\]
where \(V(t)\) denotes multiplication by a real-valued function
\(V(t,\cdot)\).

\begin{lemma}[Commutator estimate for multiplication operators]
\label{lem:potential-commutator-estimate}
Let \(r\geq 0\), and let \(V\in C^{\infty}(M,\mathbb{R})\). Then
\[
[V,\Lambda^{2r}]
\in
\Psi^{2r-1}(M).
\]
Consequently,
\[
\Lambda^{-r}[V,\Lambda^{2r}]\Lambda^{-r}
\in
\Psi^{-1}(M)
\subset
\mathcal{L}(L^{2}(M)).
\]
In particular, there exists a constant \(C_{r,V}>0\) such that
\[
\left|
\langle u,\mathrm{i}[V,\Lambda^{2r}]u\rangle
\right|
\leq
C_{r,V}\|u\|_{r}^{2}
\]
for every \(u\in C^{\infty}(M)\).
\end{lemma}

\begin{proof}
The operator \(\Lambda^{2r}\) is a classical pseudodifferential operator of
order \(2r\), while multiplication by \(V\) is of order zero. The standard
commutator rule in the pseudodifferential calculus gives
\[
[V,\Lambda^{2r}]
\in
\Psi^{2r-1}(M).
\]
Since \(\Lambda^{-r}\in\Psi^{-r}(M)\), composition yields
\[
\Lambda^{-r}[V,\Lambda^{2r}]\Lambda^{-r}
\in
\Psi^{-r+(2r-1)-r}(M)
=
\Psi^{-1}(M).
\]
An operator of order \(-1\) is bounded on \(L^{2}(M)\). Therefore
\[
\begin{split}
\left|
\langle u,\mathrm{i}[V,\Lambda^{2r}]u\rangle
\right|
&=
\left|
\left\langle
\Lambda^{r}u,
\mathrm{i}\Lambda^{-r}[V,\Lambda^{2r}]
\Lambda^{-r}\Lambda^{r}u
\right\rangle
\right|
\\
&\leq
\left\|
\Lambda^{-r}[V,\Lambda^{2r}]\Lambda^{-r}
\right\|_{\mathcal{L}(L^{2})}
\|u\|_{r}^{2}.
\end{split}
\]
This proves the estimate.
\end{proof}

\begin{proposition}[Uniform energy stability for spectral Schrödinger
cut-offs]
\label{prop:schrodinger-uniform-stability}
Let \(I\subset\mathbb{R}\) be compact and assume that
\[
V\in C\bigl(I,C^{\infty}(M,\mathbb{R})\bigr).
\]
Define
\[
H(t)=\Delta+V(t)
\]
and
\[
H_{N}(t)=P_{N}H(t)P_{N},
\qquad
P_{N}=\mathbf{1}_{[0,N]}(H_{0}).
\]
Then, for every \(r\geq 0\), there exists a constant \(C_{I,r}>0\),
independent of \(N\), such that
\[
\|U_{N}(t,s)v\|_{r}
\leq
\exp\bigl(C_{I,r}|t-s|\bigr)\|v\|_{r}
\]
for every \(v\in\mathcal{H}_{N}\) and every \(s,t\in I\).
\end{proposition}

\begin{proof}
Since \(H_{0}\) and \(\Lambda\) are functions of \(\Delta\), one has
\[
[\Delta,\Lambda^{2r}]=0.
\]
It follows that
\[
[H(t),\Lambda^{2r}]
=
[V(t),\Lambda^{2r}].
\]
By Lemma~\ref{lem:potential-commutator-estimate},
\[
\left|
\langle u,\mathrm{i}[H(t),\Lambda^{2r}]u\rangle
\right|
\leq
c_{r}(t)\|u\|_{r}^{2},
\]
where
\[
c_{r}(t)
=
\left\|
\Lambda^{-r}[V(t),\Lambda^{2r}]\Lambda^{-r}
\right\|_{\mathcal{L}(L^{2})}.
\]
The continuity of \(t\mapsto V(t)\) in \(C^{\infty}(M)\), together with
continuity of the pseudodifferential symbolic operations, implies that
\(c_{r}\) is bounded on \(I\). The conclusion follows from
Proposition~\ref{prop:uniform-cutoff-stability}.
\end{proof}

The preceding estimate verifies the main non-trivial hypothesis in the
constructive spectral approximation theorem.

\begin{theorem}[Construction and quantitative approximation for
time-dependent Schrödinger operators]
\label{thm:schrodinger-construction-rate}
Let \(M\), \(H_{0}\), and \(H(t)\) be as above, with
\[
V\in C\bigl(I,C^{\infty}(M,\mathbb{R})\bigr).
\]
Then the limits
\[
U(t,s)u
=
\lim_{N\to\infty}U_{N}(t,s)P_{N}u
\]
exist in \(L^{2}(M)\), uniformly for \((t,s)\in I\times I\), and define a
strongly continuous unitary propagator.

The propagator preserves every Sobolev space \(H^{r}(M)\). For
\(u\in C^{\infty}(M)\), it satisfies
\[
\mathrm{i}\partial_{t}U(t,s)u
=
H(t)U(t,s)u.
\]

Moreover, for every \(\sigma>0\), there exists
\(C_{I,\sigma}>0\) such that
\[
\sup_{s,t\in I}
\|U_{N}(t,s)P_{N}u-U(t,s)u\|_{L^{2}}
\leq
C_{I,\sigma}(1+N)^{-\sigma}
\|u\|_{H^{2+\sigma}}
\]
for every \(u\in H^{2+\sigma}(M)\).
More generally, for every \(r\geq 0\),
\[
\sup_{s,t\in I}
\|U_{N}(t,s)P_{N}u-U(t,s)u\|_{H^{r}}
\leq
C_{I,r,\sigma}(1+N)^{-\sigma}
\|u\|_{H^{r+2+\sigma}}.
\]
\end{theorem}

\begin{proof}
The family \(H(t)=\Delta+V(t)\) satisfies
\[
H(t):H^{r+2}(M)\longrightarrow H^{r}(M)
\]
continuously and uniformly for \(t\in I\). Thus the loss parameter in the
abstract construction is
\[
\mu=2.
\]
Proposition~\ref{prop:schrodinger-uniform-stability} gives uniform stability
of the cut-off propagators at every level of the Sobolev scale. The existence
of the limiting propagator and its strong equation on \(C^{\infty}(M)\)
therefore follow from
Theorem~\ref{thm:construction-by-spectral-cutoff}.

Applying the quantitative Cauchy estimate at level \(r\), with
\(u\in H^{r+2+\sigma}(M)\), gives
\[
\sup_{s,t\in I}
\|U_{N}(t,s)P_{N}u-U(t,s)u\|_{H^{r}}
\leq
C_{I,r,\sigma}(1+N)^{-\sigma}
\|u\|_{H^{r+2+\sigma}}.
\]
The case \(r=0\) gives the stated \(L^{2}\)-estimate.
\end{proof}

\begin{remark}
\label{rem:schrodinger-no-prior-propagator}
Unlike the corresponding statement in the preceding version of the paper,
Theorem~\ref{thm:schrodinger-construction-rate} does not assume in advance
that the non-autonomous Schrödinger equation has a propagator preserving the
Sobolev scale. Both the propagator and its Sobolev stability are obtained
from the uniform spectral cut-off estimates.
\end{remark}

\subsection{First-order pseudodifferential Hamiltonians}
\label{subsec:first-order-pseudodifferential-hamiltonians}

We next consider a Hermitian vector bundle \(E\to M\) over a compact manifold.
Let
\[
\mathcal{H}=L^{2}(M,E),
\]
and let
\[
H_0\in\Psi^1(M,E)
\]
be a positive elliptic self-adjoint operator with scalar principal
symbol, and set
\[
\Lambda=1+H_0.
\]
We use the scale
\[
\Hilb^r=D(\Lambda^r)\simeq H^r(M,E),
\qquad
P_N=\mathbf{1}_{[0,N]}(H_0).
\]

Let
\[
H(t)\in\Psi^{1}(M,E)
\]
be a continuous family of symmetric classical pseudodifferential operators.

\begin{lemma}[Energy commutator estimate in order one]
\label{lem:first-order-energy-commutator}
For every \(r\geq 0\),
\[
[H(t),\Lambda^{2r}]
\in
\Psi^{2r}(M,E).
\]
Consequently,
\[
\Lambda^{-r}[H(t),\Lambda^{2r}]\Lambda^{-r}
\in
\Psi^{0}(M,E),
\]
and there exists \(c_{r}(t)\geq 0\) such that
\[
\left|
\langle u,\mathrm{i}[H(t),\Lambda^{2r}]u\rangle
\right|
\leq
c_{r}(t)\|u\|_{r}^{2}
\]
for every \(u\in C^{\infty}(M,E)\).
If \(t\mapsto H(t)\) is continuous in the Fréchet topology of
\(\Psi^{1}(M,E)\), then \(c_{r}\) may be chosen bounded on compact time
intervals.
\end{lemma}

\begin{proof}
The commutator of operators of orders \(1\) and \(2r\) has order at most
\[
1+2r-1=2r.
\]
Therefore
\[
\Lambda^{-r}[H(t),\Lambda^{2r}]\Lambda^{-r}
\in
\Psi^{-r+2r-r}(M,E)
=
\Psi^{0}(M,E).
\]
The \(L^{2}\)-boundedness theorem for order-zero pseudodifferential operators
gives the quadratic-form estimate. Uniformity in \(t\) follows from
continuity of the relevant symbol seminorms on the compact interval \(I\).
\end{proof}

\begin{theorem}[Spectral construction for first-order Hamiltonians]
\label{thm:first-order-pseudodifferential-construction}
Let
\[
H\in C\bigl(I,\Psi^{1}(M,E)\bigr)
\]
be a family of symmetric classical pseudodifferential operators. Assume that
the continuity holds in the Fréchet topology of
\(\Psi^{1}(M,E)\). Then the spectral cut-off propagators
\[
U_{N}(t,s)
\]
associated with
\[
H_{N}(t)=P_{N}H(t)P_{N}
\]
converge strongly and uniformly on \(I\times I\) to a strongly continuous
unitary propagator \(U(t,s)\).

The propagator preserves \(H^{r}(M,E)\) for every \(r\geq 0\), and, for
\(u\in C^{\infty}(M,E)\),
\[
\mathrm{i}\partial_{t}U(t,s)u
=
H(t)U(t,s)u.
\]
Moreover, for every \(r\geq 0\) and every \(\sigma>0\),
\[
\sup_{s,t\in I}
\|U_{N}(t,s)P_{N}u-U(t,s)u\|_{H^{r}}
\leq
C_{I,r,\sigma}(1+N)^{-\sigma}
\|u\|_{H^{r+1+\sigma}}.
\]
\end{theorem}

\begin{proof}
A first-order pseudodifferential operator satisfies
\[
H(t):H^{r+1}(M,E)\longrightarrow H^{r}(M,E),
\]
uniformly for \(t\in I\). Thus the loss parameter is \(\mu=1\).
Lemma~\ref{lem:first-order-energy-commutator} verifies the uniform energy
commutator estimate at every Sobolev level. The conclusion follows from
Theorem~\ref{thm:construction-by-spectral-cutoff} and its quantitative
energy-norm version.
\end{proof}

\begin{corollary}[Dirac-type Hamiltonians]
\label{cor:dirac-type-hamiltonians}
Let \(D\) be a self-adjoint Dirac-type operator on a Hermitian Clifford module
\(E\to M\), and let
\[
V\in C\bigl(I,\Psi^{0}(M,E)\bigr)
\]
be a family of symmetric order-zero operators. Then
\[
H(t)=D+V(t)
\]
satisfies the conclusions of
Theorem~\ref{thm:first-order-pseudodifferential-construction}.
In particular, for every \(r\geq 0\) and every \(\sigma>0\),
\[
\sup_{s,t\in I}
\|U_{N}(t,s)P_{N}u-U(t,s)u\|_{H^{r}}
\leq
C_{I,r,\sigma}(1+N)^{-\sigma}
\|u\|_{H^{r+1+\sigma}}.
\]
\end{corollary}

\subsection{Higher-order Hamiltonians compatible with the reference scale}
\label{subsec:higher-order-compatible-hamiltonians}

For a general operator \(H(t)\in\Psi^{m}\) with \(m>1\), the usual symbolic
commutator estimate gives
\[
[H(t),\Lambda^{2r}]
\in
\Psi^{m+2r-1}.
\]
After conjugation by \(\Lambda^{-r}\), this has order \(m-1\), which is
generally positive. Therefore an arbitrary higher-order Hamiltonian does not
automatically satisfy the energy commutator bound used above.

The appropriate additional hypothesis is the following.

\begin{assumption}[Compatibility with the reference scale]
\label{ass:principal-symbol-compatibility}
Let \(H(t)\in\Psi^{m}(M,E)\) be symmetric. We assume that, for every
\(r\geq 0\),
\[
[H(t),\Lambda^{2r}]
\in
\Psi^{2r}(M,E)
\]
uniformly for \(t\in I\).
\end{assumption}

A sufficient symbolic condition is that the positive-order terms in the
commutator vanish. In particular, if \(h_{m}(t,x,\xi)\) and
\(\lambda_{1}(x,\xi)\) denote the principal symbols of \(H(t)\) and
\(\Lambda\), respectively, then the first required cancellation is
\[
\{h_{m}(t),\lambda_{1}\}=0.
\]
For differential Hamiltonians whose principal part is a function of the
reference elliptic operator, this compatibility is natural.

\begin{proposition}[Higher-order compatible Hamiltonians]
\label{prop:higher-order-compatible-hamiltonians}
Let
\[
H\in C\bigl(I,\Psi^{m}(M,E)\bigr)
\]
be a symmetric family satisfying
Assumption~\ref{ass:principal-symbol-compatibility}. Then the spectral cut-off
propagators are uniformly stable on every Sobolev space. The limiting
propagator exists, preserves the Sobolev scale, and satisfies
\[
\sup_{s,t\in I}
\|U_{N}(t,s)P_{N}u-U(t,s)u\|_{H^{r}}
\leq
C_{I,r,\sigma}(1+N)^{-\sigma}
\|u\|_{H^{r+m+\sigma}}
\]
for every \(r\geq 0\), every \(\sigma>0\), and every
\(u\in H^{r+m+\sigma}(M,E)\).
\end{proposition}

\begin{proof}
By the compatibility assumption,
\[
\Lambda^{-r}[H(t),\Lambda^{2r}]\Lambda^{-r}
\in
\Psi^{0}(M,E).
\]
Hence the energy commutator estimate holds uniformly on compact time
intervals. Since
\[
H(t):H^{r+m}(M,E)\longrightarrow H^{r}(M,E),
\]
the loss parameter is \(\mu=m\). The conclusion follows from the abstract
construction and quantitative convergence theorems.
\end{proof}

\begin{remark}
\label{rem:schrodinger-compatible-higher-order}
The Schrödinger family
\[
H(t)=\Delta+V(t)
\]
is a basic example of the higher-order compatible situation. Since
\(\Lambda\) is chosen as a function of \(\Delta\),
\[
[\Delta,\Lambda^{2r}]=0,
\]
and only the order-zero potential contributes to the energy commutator.
\end{remark}

\subsection{Quadratic Hamiltonians and harmonic oscillators}

The compact-manifold examples above can be complemented by a standard
non-compact class with compact resolvent: harmonic-oscillator-type
Hamiltonians on \(\mathbb R^d\).

Let
\[
        \Hilb=L^2(\mathbb R^d),
\]
and set
\[
        H_0=1-\Delta+|x|^2.
\]
The operator \(H_0\) is positive self-adjoint and has compact resolvent.
Its Hilbert scale is the standard Hermite-Sobolev scale.

Consider a time-dependent quadratic Hamiltonian
\[
        H(t)
        =
        \frac12 p^TA(t)p
        +
        \frac12 x^TB(t)x
        +
        \frac12\bigl(x^TC(t)p+p^TC(t)^Tx\bigr)
        +
        \ell(t,x,p),
\]
where
\[
        p=-\ii\nabla_x,
\]
the matrices \(A(t)\), \(B(t)\), \(C(t)\) are real and sufficiently
regular in time, and \(\ell(t,x,p)\) is at most linear in \(x\) and
\(p\).  Under the usual symmetry conditions, \(H(t)\) is a symmetric
second-order operator on the Hermite-Sobolev scale.

In this case,
\[
        H(t):\Hilb^{s+1}\to\Hilb^s
\]
if the scale is defined by \(H_0=1-\Delta+|x|^2\), since \(H_0\) itself
has differential order two.  Equivalently, if one uses
\[
        \widetilde H_0=(1-\Delta+|x|^2)^{1/2},
\]
then the loss is \(\mu=2\).

\begin{proposition}
Time-dependent quadratic Hamiltonians on \(\mathbb R^d\), controlled by
the harmonic oscillator, fit the spectral cut-off framework with
respect to the Hermite-Sobolev scale, provided the associated
non-autonomous Schr\"odinger equation is well posed.
\end{proposition}

\begin{proof}
Polynomial expressions of degree at most two in \(x\) and \(p\) are
controlled by the harmonic oscillator.  More precisely, each such
operator maps the Hermite-Sobolev space of order \(s+1\) associated
with \(H_0=1-\Delta+|x|^2\) continuously to the space of order \(s\).
The spectral projections of \(H_0\) have finite rank because the
harmonic oscillator has discrete spectrum with finite multiplicities.
The conclusion follows from the abstract convergence theorem.
\end{proof}

\begin{remark}
Quadratic Hamiltonians are especially useful for comparison with the
usual explicit Feynman integral formulae, because their propagators are
metaplectic operators and their phases are quadratic actions.
\end{remark}

\subsection{Log-polyhomogeneous pseudodifferential Hamiltonians}

We now consider a class which is important for renormalized traces.
Let \(M\) be compact and let
\[
        H(t)\in \Psi_{\log}^{m,k}(M,E)
\]
be a family of log-polyhomogeneous pseudodifferential operators.  In
local coordinates, this means that the symbol has an asymptotic
expansion of the form
\[
        a(t,x,\xi)
        \sim
        \sum_{j=0}^{\infty}
        \sum_{\ell=0}^{k_j}
        a_{m-j,\ell}(t,x,\xi)(\log|\xi|)^\ell,
\]
where \(a_{m-j,\ell}\) is homogeneous of degree \(m-j\) in \(\xi\) for
large \(|\xi|\).

The logarithmic factors do not change the principal power order, but
they affect the exact Sobolev mapping estimates.  A convenient
formulation is the following: for every \(\varepsilon>0\),
\[
        H(t):H^{s+m+\varepsilon}(M,E)\to H^s(M,E).
\]
Equivalently, one may work with logarithmic Sobolev spaces and keep the
loss exactly equal to \(m\).

\begin{proposition}
Let \(H(t)\) be a continuous family of symmetric log-polyhomogeneous
pseudodifferential operators of order \(m\) on a compact manifold.  Then
for every \(\varepsilon>0\),
\[
        H(t):H^{s+m+\varepsilon}(M,E)\to H^s(M,E)
\]
continuously and uniformly for \(t\) in compact intervals.  Therefore
the spectral cut-off convergence theorem applies with loss
\[
        \mu=m+\varepsilon
\]
for arbitrary \(\varepsilon>0\), assuming well-posedness of the full
evolution.
\end{proposition}

\begin{proof}
Logarithmic powers satisfy, for every \(\varepsilon>0\), an estimate of
the form
\[
        (\log\langle \xi\rangle)^k
        \leq
        C_{\varepsilon,k}\langle \xi\rangle^\varepsilon.
\]
Thus a log-polyhomogeneous symbol of order \(m\) is bounded by a
classical symbol of order \(m+\varepsilon\).  The standard
pseudodifferential mapping theorem then gives
\[
        H(t):H^{s+m+\varepsilon}(M,E)\to H^s(M,E).
\]
Uniformity in \(t\) follows from continuity of the symbol family in the
log-polyhomogeneous Fréchet topology.  The conclusion follows from the
abstract theorem.
\end{proof}

\begin{remark}
This class is not needed for the basic construction of the solution,
but it becomes natural when one studies regularized traces.  In the
log-polyhomogeneous calculus, zeta functions may have higher-order
poles, and residue-type functionals are richer than in the classical
case.
\end{remark}

\subsection{Smoothing and negative-order Hamiltonians}

Let
\[
        H(t)\in \Psi^{-\rho}(M,E),
        \qquad \rho>0,
\]
be a family of symmetric negative-order pseudodifferential operators on
a compact manifold.  Then
\[
        H(t):H^s(M,E)\to H^{s+\rho}(M,E).
\]
In particular,
\[
        H(t):H^s(M,E)\to H^s(M,E)
\]
is bounded for every \(s\).  Thus one may take
\[
        \mu=0.
\]

\begin{proposition}
A continuous family of symmetric pseudodifferential operators of
negative order satisfies the cut-off convergence assumptions with no
loss of derivatives, provided the corresponding unitary propagator is
well defined.
\end{proposition}

\begin{proof}
Negative-order pseudodifferential operators are bounded on every
Sobolev space and even regularizing by \(\rho\) derivatives.  Therefore
the mapping estimate
\[
        H(t):\Hilb^s\to\Hilb^s
\]
holds uniformly for \(t\in I\).  The convergence theorem applies with
\(\mu=0\).
\end{proof}

\begin{remark}
If the order is less than \(-\dim M\), then such operators are
trace-class.  These examples are useful as test cases for comparing
ordinary traces with cut-off or renormalized traces.
\end{remark}

\subsection{Abstract Kato--Rellich perturbations}

Finally, we record an abstract class which does not rely on a manifold.

Let \(A\) be a positive self-adjoint operator with compact resolvent on
\(\Hilb\), and set
\[
        H_0=1+A.
\]
Suppose that
\[
        H(t)=A+V(t),
\]
where \(V(t)\) is symmetric and \(A\)-bounded with relative bound
strictly less than \(1\), uniformly for \(t\in I\).  That is, there
exist constants \(a<1\) and \(b\geq0\) such that
\[
        \|V(t)u\|
        \leq
        a\|Au\|+b\|u\|
\]
for all \(u\in D(A)\) and all \(t\in I\).

By the Kato--Rellich theorem, each \(H(t)=A+V(t)\) is self-adjoint on
the common domain \(D(A)\).  If \(t\mapsto V(t)\) is sufficiently
regular in the appropriate operator topology, the non-autonomous
Schr\"odinger equation is well posed by the standard theory of
time-dependent self-adjoint Hamiltonians.

\begin{proposition}
Let \(A\) be positive self-adjoint with compact resolvent, and let
\[
        H(t)=A+V(t)
\]
be a uniformly relatively \(A\)-bounded symmetric perturbation with
relative bound \(<1\).  Assume that \(t\mapsto V(t)\) is regular enough
to generate a unitary non-autonomous propagator preserving the scale of
\(A\).  Then the spectral cut-off construction applies with respect to
\(H_0=1+A\).
\end{proposition}

\begin{proof}
Since \(H_0=1+A\), the space \(\Hilb^1\) is \(D(A)\), up to equivalent
norms.  The estimate
\[
        \|V(t)u\|
        \leq
        a\|Au\|+b\|u\|
\]
implies that
\[
        H(t)=A+V(t):\Hilb^1\to\Hilb
\]
is bounded uniformly for \(t\in I\).  Hence the abstract mapping
assumption holds with \(\mu=1\), relative to the scale generated by
\(H_0=1+A\).  Compactness of the resolvent gives finite-dimensional
spectral cut-offs.  The convergence theorem applies once the
well-posedness and stability of the full propagator are assumed.
\end{proof}

\begin{remark}
This abstract class is useful because it includes many concrete
Hamiltonians without using pseudodifferential calculus explicitly.  It
also clarifies the role of the reference operator \(H_0\): the cut-off
is not tied to a coordinate representation, but to an energy scale.
\end{remark}

\subsection{Summary of the examples}

The examples above may be summarized as follows.

\[
\begin{array}{c|c|c}
\text{Class of Hamiltonians}
&
\text{Natural reference operator}
&
\text{Typical loss}
\\ \hline
-\Delta+V(t)
&
(1+\Delta)^{1/2}
&
\mu=2
\\
\text{Differential operators of order }m
&
(1+\Delta_E)^{1/2}
&
\mu=m
\\
\Psi_{\mathrm{cl}}^m(M,E)
&
(1+\Delta_E)^{1/2}
&
\mu=m
\\
D+V(t)
&
1+|D|
&
\mu=1
\\
\text{Quadratic Hamiltonians on }\mathbb R^d
&
1-\Delta+|x|^2
&
\mu=1
\\
\Psi_{\log}^{m,k}(M,E)
&
(1+\Delta_E)^{1/2}
&
\mu=m+\varepsilon
\\
\Psi^{-\rho}(M,E)
&
(1+\Delta_E)^{1/2}
&
\mu=0
\\
A+V(t)\text{, Kato--Rellich}
&
1+A
&
\mu=1
\end{array}
\]

All these classes are compatible with the spectral cut-off construction
provided the corresponding non-autonomous Hamiltonian equation is
well posed and the full propagator preserves the required Hilbert
scale.  The examples differ mainly in how the mapping estimates and the
energy stability estimates are proved: by elliptic regularity for
differential operators, by symbolic calculus for pseudodifferential
operators, by the Hermite calculus for quadratic Hamiltonians, and by
relative boundedness estimates in the abstract Kato--Rellich setting.
\appendix

\section{Effective Hamiltonians in periodically driven quantum systems}
\label{app:effective-hamiltonians}

The purpose of this appendix is to recall the role of effective
Hamiltonians in periodically driven quantum systems, and to explain why
they naturally appear in connection with the spectral cut-off procedure
developed in the main body of the paper.  No result of this appendix is
needed for the proof of the convergence theorem.  Rather, the appendix
provides the physical and mathematical background for the discussion of
periodic Hamiltonians, Floquet--Magnus expansions, and cut-off dynamics.

\subsection{Non-autonomous Schr\"odinger equations}

Let \(\Hilb\) be a complex Hilbert space.  A time-dependent quantum system
is described, at least formally, by a non-autonomous Schr\"odinger
equation
\[
        \ii \partial_t u(t)=H(t)u(t),
        \qquad u(s)=u_s,
\]
where \(H(t)\) is a family of self-adjoint, or essentially self-adjoint,
operators on \(\Hilb\).  If the equation is well posed, it generates a
two-parameter unitary propagator
\[
        U(t,s):\Hilb\to\Hilb,
\]
satisfying
\[
        U(t,r)U(r,s)=U(t,s),
        \qquad
        U(s,s)=\Id,
\]
and, on a suitable common domain,
\[
        \ii \partial_t U(t,s)u=H(t)U(t,s)u.
\]
In the autonomous case \(H(t)=H\), the propagator is simply
\[
        U(t,s)=\ee^{-\ii(t-s)H}.
\]
For non-autonomous Hamiltonians, no such simple exponential formula is
available in general.  One often writes formally
\[
        U(t,s)
        =
        \mathcal T
        \exp\left(
        -\ii\int_s^t H(\tau)\dd \tau
        \right),
\]
where \(\mathcal T\) denotes time ordering.  This notation is useful in
physics, but it hides serious domain questions when the operators
\(H(t)\) are unbounded.

The main construction of this paper replaces this formal object by
finite-dimensional spectral cut-offs.  If \(H_0\) is a positive
self-adjoint reference operator and
\[
        P_N=\mathbf 1_{[0,N]}(H_0),
\]
then the cut-off Hamiltonians
\[
        H_N(t)=P_NH(t)P_N
\]
act on the finite-dimensional spaces \(P_N\Hilb\), when \(H_0\) has
compact resolvent.  Their propagators \(U_N(t,s)\) are ordinary
finite-dimensional unitary matrices, and may be represented by
time-sliced oscillatory integrals.  The analytical problem is then to
prove
\[
        U_N(t,s)P_Nu \longrightarrow U(t,s)u
\]
strongly, uniformly for \(t,s\) in compact intervals.

\subsection{Periodic Hamiltonians and the Floquet problem}

Suppose now that the Hamiltonian is periodic in time.  More precisely,
let \(T>0\) and assume
\[
        H(t+T)=H(t).
\]
The propagator over one period,
\[
        U(T,0),
\]
is called the monodromy operator.  In finite dimension, Floquet theory
implies that, after choosing a branch of the logarithm, one may write
\[
        U(T,0)=\ee^{-\ii T H_F},
\]
where \(H_F\) is a time-independent operator, called a Floquet
Hamiltonian.  It is not unique, because the logarithm of a unitary
operator is not unique.  Nevertheless, once a branch has been chosen,
the dynamics at integer multiples of the period is described by
\[
        U(qT,0)=\ee^{-\ii qT H_F},
        \qquad q\in\mathbb Z.
\]
Thus the periodically driven system admits an autonomous description at
stroboscopic times.

In applications, one is often interested in the high-frequency regime,
that is, in the limit \(T\to0\).  It is then convenient to start from a
\(1\)-periodic Hamiltonian \(H(\tau)\) and define the \(T\)-periodic
rescaling
\[
        H^{(T)}(t)=H(t/T).
\]
The corresponding equation is
\[
        \ii\partial_t u(t)=H^{(T)}(t)u(t).
\]
The central question is whether the true dynamics may be approximated,
for small \(T\), by an autonomous evolution
\[
        \ee^{-\ii t H_{\mathrm{eff}}^{(T)}},
\]
where \(H_{\mathrm{eff}}^{(T)}\) is an effective Hamiltonian admitting an
asymptotic expansion in powers of \(T\).

\subsection{The Floquet--Magnus expansion}

In finite dimension, and more generally for bounded operators under
suitable convergence conditions, the effective Hamiltonian may be
constructed from the Magnus expansion.  The propagator is written in the
form
\[
        U^{(T)}(t,0)=\ee^{\Omega^{(T)}(t)},
\]
where \(\Omega^{(T)}(t)\) is expanded as a series of iterated integrals
and commutators.  At one period, one defines
\[
        H_{\mathrm{FM}}^{(T)}
        =
        \frac{\ii}{T}\Omega^{(T)}(T).
\]
The first terms are
\[
        H_{\mathrm{FM}}^{(T)}
        \sim
        H_{\mathrm{FM},0}^{(T)}
        +
        H_{\mathrm{FM},1}^{(T)}
        +
        H_{\mathrm{FM},2}^{(T)}
        +\cdots ,
\]
with
\[
        H_{\mathrm{FM},0}^{(T)}
        =
        \frac1T\int_0^T H^{(T)}(t_1)\dd t_1,
\]
and
\[
        H_{\mathrm{FM},1}^{(T)}
        =
        \frac{\ii}{2T}
        \int_0^T
        \int_0^{t_1}
        [H^{(T)}(t_1),H^{(T)}(t_2)]
        \dd t_2\,\dd t_1.
\]
Higher-order terms involve nested commutators of increasing length.
Equivalently, after the change of variable \(t=T\tau\), the expansion
takes the form
\[
        H_{\mathrm{FM}}^{(T)}
        \sim
        H_{\mathrm{eff}}^{[0]}
        +
        T H_{\mathrm{eff}}^{[1]}
        +
        T^2 H_{\mathrm{eff}}^{[2]}
        +\cdots .
\]
The leading term is the time average
\[
        H_{\mathrm{eff}}^{[0]}
        =
        \int_0^1 H(\tau)\dd\tau.
\]
The next term is proportional to the averaged commutator
\[
        H_{\mathrm{eff}}^{[1]}
        =
        \frac{\ii}{2}
        \int_0^1
        \int_0^{\tau_1}
        [H(\tau_1),H(\tau_2)]
        \dd\tau_2\,\dd\tau_1.
\]
These formulae explain why effective Hamiltonians are sensitive not
only to the average of the driving, but also to the non-commutativity of
the Hamiltonian at different times.

\begin{remark}
The appearance of commutators is one of the reasons why effective
Hamiltonians are useful in physics.  A rapidly driven system may acquire
effective interactions which are absent from the instantaneous
Hamiltonian itself.  This mechanism is one of the basic principles of
Floquet engineering.
\end{remark}

\subsection{Why the unbounded case is delicate}

For bounded operators, the above formulae are meaningful as identities
or asymptotic expansions in operator norm, at least under suitable
smallness assumptions.  For unbounded Hamiltonians, several new
difficulties appear.

First, products and commutators such as
\[
        H(t_1)H(t_2),
        \qquad
        [H(t_1),H(t_2)]
\]
need not have a common dense domain.  Even when they are defined on a
common core, it is not automatic that the resulting operators are
symmetric, essentially self-adjoint, or closable.

Second, the logarithm of the monodromy operator \(U(T,0)\) may be a
self-adjoint operator, but this does not imply that the formal
Floquet--Magnus coefficients define self-adjoint operators.

Third, convergence of the full Magnus series is not the correct issue
for many applications.  One often needs only finite-order effective
Hamiltonians, together with estimates showing that the associated
autonomous dynamics approximates the true propagator in the
high-frequency regime.

Recent work on the Floquet--Magnus expansion for unbounded operators
addresses precisely these problems.  The strategy is to work on a space
of smooth vectors associated with a reference self-adjoint operator
\(H_0\),
\[
        C^\infty(H_0)=\bigcap_{m\geq1}D(H_0^m),
\]
and to impose estimates showing that both \(H(t)\) and the propagator
are controlled with respect to the \(H_0\)-scale.  Under additional
spectral and symmetry assumptions, finite-order effective Hamiltonians
can be defined on \(C^\infty(H_0)\), shown to coincide with the
Floquet--Magnus coefficients, and proved to admit self-adjoint
extensions; see, for example,
\cite{BurgarthHillierLonigroRichter2026}.

\subsection{Spectral cut-offs and effective Hamiltonians}

The spectral cut-off approach gives a natural way of connecting
finite-dimensional Floquet theory with unbounded Hamiltonian dynamics.
Let
\[
        P_N=\mathbf 1_{[0,N]}(H_0)
\]
and define
\[
        H_N^{(T)}(t)=P_NH^{(T)}(t)P_N.
\]
On the finite-dimensional space \(P_N\Hilb\), the propagator
\(U_N^{(T)}(t,s)\) is generated by a bounded matrix-valued Hamiltonian.
Therefore the usual finite-dimensional Floquet--Magnus construction
applies.  One obtains a cut-off effective Hamiltonian
\[
        H_{\mathrm{FM},L,N}^{(T)}
\]
at every finite order \(L\).

The guiding question is whether these cut-off effective Hamiltonians
approximate, in a suitable sense, an effective Hamiltonian for the
unbounded dynamics.  One expects convergence of the form
\[
        H_{\mathrm{FM},L,N}^{(T)}P_Nu
        \longrightarrow
        H_{\mathrm{FM},L}^{(T)}u,
        \qquad
        u\in C^\infty(H_0),
\]
or, more cautiously, convergence in the sense of quadratic forms or in
the topology of an \(H_0\)-Sobolev scale.  The exact topology depends on
the order of the Hamiltonians and on the strength of the estimates
available for the commutators.

At the level of propagators, the expected compatibility statement is
\[
        \ee^{-\ii qT H_{\mathrm{FM},L,N}^{(T)}}P_Nu
        \longrightarrow
        \ee^{-\ii qT \widehat H_{\mathrm{FM},L}^{(T)}}u
\]
for stroboscopic times \(t=qT\), whenever
\(\widehat H_{\mathrm{FM},L}^{(T)}\) is a self-adjoint extension of the
formal effective Hamiltonian.  This should be compared with the
convergence of the exact cut-off dynamics
\[
        U_N^{(T)}(qT,0)P_Nu
        \longrightarrow
        U^{(T)}(qT,0)u.
\]
Thus the spectral cut-off construction provides a bridge between three
objects:
\[
\begin{array}{ccc}
\text{cut-off oscillatory integrals}
& \longrightarrow &
\text{cut-off propagators}
\\[1mm]
&& \downarrow
\\[-1mm]
&&
\text{cut-off effective Hamiltonians}
\\[1mm]
&& \downarrow N\to\infty
\\[-1mm]
&&
\text{unbounded effective Hamiltonians}.
\end{array}
\]

\subsection{Relation with oscillatory integrals}

For each fixed cut-off \(N\), the propagator \(U_N(t,s)\) may be
represented by a time-slicing procedure.  If
\[
        s=t_0<t_1<\cdots<t_M=t,
\]
then
\[
        U_{N,M}(t,s)
        =
        \prod_{j=M-1}^{0}
        \exp\left(
        -\ii(t_{j+1}-t_j)H_N(t_j)
        \right)
\]
converges to \(U_N(t,s)\) as the mesh of the partition tends to zero.
After choosing coordinates on \(P_N\Hilb\) and inserting
finite-dimensional resolutions of the identity, this product can be
written as a finite-dimensional oscillatory integral
\[
        I_{N,M}(t,s)u
        =
        \int
        \ee^{\ii S_{N,M}}
        A_{N,M}
        u_N
        ,
\]
where \(u_N=P_Nu\).  The phase \(S_{N,M}\) is a discrete action and
\(A_{N,M}\) is the corresponding amplitude.

The construction used in the present paper may therefore be understood
as a two-step replacement of the formal real-time Feynman integral:
\[
        \int \ee^{\ii S[q]}\mathcal Dq.
\]
First, the spectral projection \(P_N\) produces a genuine
finite-dimensional oscillatory integral.  Second, the limit
\(N\to\infty\) is taken in the strong topology of the Hilbert space, or
in a stronger topology associated with the reference operator \(H_0\).
Thus no infinite-dimensional oscillatory measure is postulated.

In the periodic case, these cut-off oscillatory integrals also carry
the finite-dimensional Floquet--Magnus structure.  Consequently, the
effective Hamiltonian can be viewed as a high-frequency autonomous
description extracted from the cut-off oscillatory dynamics before the
limit \(N\to\infty\) is taken.

\section{Spectral cut-off traces and renormalized real-time amplitudes}
\label{app:cutoff-traces}

The aim of this appendix is to record a possible trace-theoretic
extension of the spectral cut-off construction developed in the main
text.  The convergence theorem proved in the paper concerns vectors and
propagators in the strong topology.  Trace expressions are more
singular.  Even when the propagator \(U(t,s)\) is a well-defined unitary
operator on \(\Hilb\), it is usually not trace-class.  Thus expressions
such as
\[
        \Tr(U(t,s))
\]
are not defined in general.

The spectral cut-off method provides a natural way of assigning finite
approximations to such quantities.  If
\[
        P_N=\mathbf 1_{[0,N]}(H_0)
\]
is the spectral projection associated with the reference operator
\(H_0\), then
\[
        \Tr(P_NU(t,s)P_N)
\]
is finite whenever \(P_N\Hilb\) is finite-dimensional.  The question is
whether this family admits an asymptotic expansion as \(N\to\infty\),
and whether a finite part may be extracted.

\subsection{Spectral cut-off traces}

Assume in this subsection that \(H_0\) is a positive self-adjoint
operator with compact resolvent.  Let
\[
        P_N=\mathbf 1_{[0,N]}(H_0).
\]
For a bounded operator \(A\) on \(\Hilb\), we define its \(H_0\)-spectral
cut-off trace by
\[
        \Tr_N(A)
        :=
        \Tr(P_NAP_N).
\]
This is a finite-dimensional trace.

\begin{definition}[Spectral cut-off trace]
Let \(A\in\mathcal L(\Hilb)\).  The spectral cut-off trace of \(A\) at
level \(N\) is
\[
        \Tr_N(A)=\Tr(P_NAP_N).
\]
If, as \(N\to\infty\), \(\Tr_N(A)\) admits an asymptotic expansion of
the form
\[
        \Tr_N(A)
        \sim
        \sum_{j=0}^{J} c_j N^{\alpha_j}
        +
        c_{\log}\log N
        +
        c_0
        +
        o(1),
\]
with \(\alpha_j>0\), then we define the finite-part cut-off trace by
\[
        \Tr_{\mathrm{fp}}^{H_0}(A)
        :=
        c_0.
\]
\end{definition}

This definition is intentionally formal at this level.  In concrete
settings, for example for classical pseudodifferential operators on a
compact manifold, the existence of such expansions follows from
standard symbolic and spectral asymptotic techniques.  In an abstract
Hilbert space, additional assumptions are necessary.
The use of finite parts and cut-off procedures is closely related to
the construction of normalized means for infinite-volume objects.  In
particular, mean values associated with infinite volume measures,
infinite products, and heuristic infinite-dimensional Lebesgue measures
were studied in \cite{Magnot2017MeanValue}.  From this point of view,
the cut-off trace
\[
        \Tr_N(A)=\Tr(P_NAP_N)
\]
may be regarded as an operator-theoretic analogue of an exhaustion
procedure: one first assigns a finite value to the truncated object and
then studies the existence of a limit, a normalized limit, or a finite
part as the cut-off is removed.
\begin{remark}
The finite part depends in general on the choice of the reference
operator \(H_0\) and on the precise form of the cut-off.  This
dependence is not a defect of the construction; it is part of the
renormalization data.  One should distinguish between quantities which
are cut-off dependent and local anomaly terms which may have a more
intrinsic meaning.
\end{remark}

For the cut-off propagators constructed in the main text, one may
consider
\[
        \Tr\big(U_N(t,s)\big),
\]
where \(U_N(t,s)\) is the finite-dimensional propagator associated with
\[
        H_N(t)=P_NH(t)P_N.
\]
This differs slightly from \(\Tr_N(U(t,s))\), since \(U_N(t,s)\) is not
in general equal to \(P_NU(t,s)P_N\).  However, under convergence
assumptions on the cut-off dynamics, the two quantities may be expected
to have the same finite part, up to possible defect terms.

\subsection{Heat regularization}

A smoother alternative to the sharp spectral cut-off is the heat
regularization associated with \(H_0\).  For \(\varepsilon>0\), define
\[
        \Tr_\varepsilon(A)
        :=
        \Tr(\ee^{-\varepsilon H_0}A),
\]
whenever the product \(\ee^{-\varepsilon H_0}A\) is trace-class.

\begin{definition}[Heat-regularized finite part]
Let \(A\in\mathcal L(\Hilb)\).  Assume that
\(\ee^{-\varepsilon H_0}A\) is trace-class for every
\(\varepsilon>0\), and that
\[
        \Tr(\ee^{-\varepsilon H_0}A)
\]
admits an asymptotic expansion as \(\varepsilon\to0^+\) of the form
\[
        \Tr(\ee^{-\varepsilon H_0}A)
        \sim
        \sum_{j=0}^{J} a_j\varepsilon^{-\beta_j}
        +
        a_{\log}\log\varepsilon
        +
        a_0
        +
        o(1),
        \qquad \beta_j>0.
\]
We define
\[
        \Tr_{\mathrm{ren}}^{H_0}(A)
        :=
        a_0.
\]
\end{definition}

The heat regularization is often better adapted to evolution equations
than the sharp cut-off \(P_N\).  It is also closer to the
Feynman--Kac point of view: the factor
\[
        \ee^{-\varepsilon H_0}
\]
introduces an infinitesimal imaginary-time damping, whereas the
unitary propagator
\[
        U(t,s)
\]
carries the real-time oscillatory dynamics.

For example, one may consider
\[
        Z_\varepsilon(t,s)
        :=
        \Tr\big(\ee^{-\varepsilon H_0}U(t,s)\big).
\]
Formally, this is a heat-regularized real-time amplitude.  In path
integral language, it corresponds to an oscillatory weight damped by a
short imaginary-time contribution.

\subsection{Closed-path amplitudes}

For fixed \(N\), the trace of the cut-off propagator has a simple
oscillatory interpretation.  Indeed, the time-sliced approximation
\[
        U_{N,M}(t,s)
        =
        \prod_{j=M-1}^{0}
        \exp\left(
        -\ii(t_{j+1}-t_j)H_N(t_j)
        \right)
\]
may be represented, after inserting finite-dimensional resolutions of
the identity, as an oscillatory integral.  Taking the trace imposes a
closed-path condition.

Thus, at the formal level,
\[
        \Tr(U_{N,M}(t,s))
        =
        \int_{\text{closed cut-off paths}}
        \ee^{\ii S_{N,M}}
        A_{N,M}.
\]
Here \(S_{N,M}\) is the discrete action associated with the cut-off
Hamiltonian, and \(A_{N,M}\) is the corresponding amplitude.

The limiting procedure is therefore
\[
        \Tr(U_N(t,s))
        =
        \lim_{M\to\infty}
        \Tr(U_{N,M}(t,s)),
\]
followed, if possible, by extraction of a finite part as \(N\to\infty\):
\[
        \Tr_{\mathrm{fp}}^{H_0}(U(t,s))
        :=
        \FP_{N\to\infty}\Tr(U_N(t,s)).
\]
This formula should not be read as a definition valid in complete
generality.  Rather, it describes the guiding principle: the
renormalized real-time amplitude is obtained from closed cut-off
oscillatory paths.

\subsection{Comparison between sharp and heat cut-offs}

The sharp spectral cut-off and the heat cut-off lead to two families of
regularized traces:
\[
        \Tr_N(A)=\Tr(P_NAP_N),
\]
and
\[
        \Tr_\varepsilon(A)=\Tr(\ee^{-\varepsilon H_0}A).
\]
Under suitable spectral assumptions, these two regularizations are
related by Tauberian principles.  However, their finite parts need not
coincide without additional normalization.

For the purposes of the present paper, the sharp cut-off has the
advantage of being directly compatible with the Galerkin approximation
used to construct the dynamics.  The heat cut-off has the advantage of
being smoother and better adapted to analytic continuation and
Feynman--Kac type formulae.

In concrete applications, one may therefore study both quantities:
\[
        \FP_{N\to\infty}\Tr(U_N(t,s)),
\]
and
\[
        \FP_{\varepsilon\to0^+}
        \Tr(\ee^{-\varepsilon H_0}U(t,s)).
\]
The comparison between them is itself a renormalization problem.

\subsection{Zeta regularization}

A third regularization, closely related to the heat cut-off, is obtained
from complex powers of \(H_0\).  Let \(Q=1+H_0\), or more generally let
\(Q\) be a positive elliptic reference operator.  For \(\Re z\) large,
one may consider
\[
        \zeta_A(z)
        :=
        \Tr(AQ^{-z}).
\]
If \(\zeta_A(z)\) admits a meromorphic continuation to a neighbourhood
of \(z=0\), one may define the \(Q\)-renormalized trace by
\[
        \Tr_{\mathrm{ren}}^Q(A)
        :=
        \FP_{z=0}\Tr(AQ^{-z}).
\]

\begin{definition}[Zeta-renormalized trace]
Let \(Q\) be a positive self-adjoint reference operator.  Suppose that
\[
        z\mapsto \Tr(AQ^{-z})
\]
is holomorphic for \(\Re z\) sufficiently large and admits a meromorphic
continuation to a neighbourhood of \(0\).  The \(Q\)-renormalized trace
of \(A\) is
\[
        \Tr_{\mathrm{ren}}^Q(A)
        :=
        \FP_{z=0}\Tr(AQ^{-z}).
\]
\end{definition}

In pseudodifferential settings, this construction is closely related to
the Kontsevich--Vishik canonical trace, to weighted traces, and to the
noncommutative residue.  It is also the natural framework in which
trace-defect formulae for commutators appear.

\subsection{Trace defects and commutators}

One of the main reasons renormalized traces are relevant for effective
Hamiltonians is the commutator structure of the Floquet--Magnus
expansion.  At first order beyond the average, the effective
Hamiltonian contains terms of the form
\[
        \int_0^1\int_0^{t_1}
        [H(t_1),H(t_2)]
        \dd t_2\,\dd t_1.
\]
Higher-order corrections contain nested commutators.

For an ordinary trace,
\[
        \Tr([A,B])=0
\]
whenever the expression is well defined and trace-class.  For a
renormalized trace, this cyclicity property generally fails:
\[
        \Tr_{\mathrm{ren}}^Q([A,B])\neq0.
\]
The defect is often local.  In classical pseudodifferential calculi, it
is governed by residue-type expressions involving \(\log Q\).  Thus
higher-order effective Hamiltonians may produce non-trivial
renormalized trace contributions even though they are built from
commutators.

This suggests that, in periodic Hamiltonian systems, the quantity
\[
        \Tr_{\mathrm{ren}}^Q
        \left(
        \ee^{-\ii t\widehat H_{\mathrm{eff},L}^{(T)}}
        \right)
\]
may contain anomaly terms associated with the high-frequency driving.
At the infinitesimal level, these terms should be visible in expressions
such as
\[
        \Tr_{\mathrm{ren}}^Q
        \left(
        H_{\mathrm{eff},L}^{(T)}
        \right),
\]
or in regularized derivatives of the corresponding partition functions.

\begin{remark}
The term ``anomaly'' is used here in a precise trace-theoretic sense:
it refers to the failure of a regularized trace to be cyclic.  No
physical anomaly is asserted without further structural assumptions.
\end{remark}

\subsection{Effective dynamics and regularized traces}

Let \(H^{(T)}(t)=H(t/T)\) be a periodic Hamiltonian in the high-frequency
regime.  Suppose that one has constructed a finite-order effective
Hamiltonian
\[
        H_{\mathrm{eff},L}^{(T)}
        =
        \sum_{\ell=0}^{L}
        T^\ell H_{\mathrm{eff}}^{[\ell]},
\]
defined on a common core and admitting a self-adjoint extension
\[
        \widehat H_{\mathrm{eff},L}^{(T)}.
\]
Then one may compare the regularized exact amplitude
\[
        Z_\varepsilon^{(T)}(t)
        =
        \Tr\left(
        \ee^{-\varepsilon H_0}
        U^{(T)}(t,0)
        \right)
\]
with the regularized effective amplitude
\[
        Z_{\varepsilon,\mathrm{eff},L}^{(T)}(t)
        =
        \Tr\left(
        \ee^{-\varepsilon H_0}
        \ee^{-\ii t\widehat H_{\mathrm{eff},L}^{(T)}}
        \right).
\]
At stroboscopic times \(t=qT\), one expects the difference between these
quantities to reflect the accuracy of the effective Hamiltonian.  A
typical problem is to determine whether
\[
        \FP_{\varepsilon\to0^+}
        \left[
        Z_\varepsilon^{(T)}(qT)
        -
        Z_{\varepsilon,\mathrm{eff},L}^{(T)}(qT)
        \right]
        =
        O(T^{L+2})
\]
as \(T\to0\).  Such a statement does not follow directly from strong
operator convergence.  It requires additional trace-class or smoothing
estimates.

This is why the trace problem is analytically more singular than the
propagator convergence problem.  The main body of the paper constructs
solutions and propagators.  The present appendix indicates how one may
then attempt to assign finite real-time amplitudes to them.

\subsection{A possible abstract assumption}

For later use, it is convenient to isolate the kind of assumption that
would make the trace theory rigorous.

\begin{assumption}[Regularized trace asymptotics]
Let \(A(t)\) be a family of bounded operators on \(\Hilb\).  We say that
\(A(t)\) has \(H_0\)-regularized trace asymptotics on a compact interval
\(I\subset\mathbb R\) if, for every \(t\in I\),
\[
        \ee^{-\varepsilon H_0}A(t)
\]
is trace-class for \(\varepsilon>0\), and if the asymptotic expansion
\[
        \Tr(\ee^{-\varepsilon H_0}A(t))
        \sim
        \sum_{j=0}^{J} a_j(t)\varepsilon^{-\beta_j}
        +
        a_{\log}(t)\log\varepsilon
        +
        a_0(t)
        +
        o(1)
\]
holds as \(\varepsilon\to0^+\), uniformly for \(t\in I\).
\end{assumption}

Under this assumption one may define
\[
        \Tr_{\mathrm{ren}}^{H_0}(A(t))=a_0(t),
\]
uniformly in time.  If \(A(t)=U(t,s)\), this gives a renormalized
real-time amplitude.  If
\[
        A(t)=\ee^{-\ii t\widehat H_{\mathrm{eff},L}^{(T)}},
\]
it gives the corresponding effective amplitude.

\subsection{Conclusion}

The spectral cut-off construction developed in the paper has two
levels.  The first level is dynamical: it constructs solutions of
non-autonomous Hamiltonian equations as limits of finite-dimensional
oscillatory integrals.  This is the part treated in the main text.

The second level is trace-theoretic: once finite-dimensional
oscillatory amplitudes have been constructed, one may take their traces,
interpret them as closed-path amplitudes, and attempt to extract finite
parts as the cut-off is removed.  This second level is not automatic.
It requires additional spectral asymptotics, smoothing estimates, or
pseudodifferential structure.

This perspective is consistent with the philosophy of generalized
means developed in \cite{Magnot2017MeanValue}: divergent global
quantities are not treated as ordinary integrals, but are approached
through prescribed finite-dimensional or finite-volume exhaustions and
renormalized limiting procedures.

In periodic systems, effective Hamiltonians provide a further
simplification.  They replace the exact time-dependent dynamics, at
stroboscopic times and in the high-frequency regime, by an autonomous
evolution.  Combining this with cut-off traces leads naturally to
renormalized effective amplitudes.  For the Schr\"odinger, first-order pseudodifferential, and compatible
higher-order classes, the commutator estimates above construct the
propagator directly and yield quantitative spectral convergence.  For
the remaining classes, the stated mapping estimates imply convergence
whenever the corresponding non-autonomous evolution preserves the
required Hilbert scale.



\end{document}